\documentclass[a4paper]{amsart}

\usepackage[all]{xy}

\newtheorem{thm}{Theorem}[section]
\newtheorem{lem}[thm]{Lemma}
\newtheorem{prop}[thm]{Proposition}
\newtheorem{cor}[thm]{Corollary}
\newtheorem{defi}[thm]{Definition}
\newtheorem{rema}[thm]{Remark}
\newtheorem*{thm*}{Theorem}

\newcommand{\bs}{\backslash}
\newcommand{\longto}{\longrightarrow}
\newcommand{\R}{\mathbb{R}}
\newcommand{\C}{\mathbb{C}}
\newcommand{\Z}{\mathbb{Z}}
\newcommand{\N}{\mathbb{N}}
\newcommand{\w}{\omega}
\newcommand{\s}{\sigma}
\newcommand{\pphi}{\varphi}
\newcommand{\taum}{\tau^{-}}
\newcommand{\qhamsp}{(M,\w,\mu:M\to U)}
\newcommand{\qhamquot}{\mu^{-1}(\{1\})/U}
\newcommand{\fibre}{\mu^{-1}(\{1\})}
\newcommand{\Walc}{\mathcal{W}}
\newcommand{\clWalc}{\overline{\mathcal{W}}}
\newcommand{\im}{\mathrm{Im}\, }
\newcommand{\fraku}{\mathfrak{u}}
\newcommand{\frakt}{\mathfrak{t}}
\newcommand{\Mb}{M^{\beta}}

\newcommand{\clWcham}{\overline{\mathfrak{t}^*_+}}
\newcommand{\calV}{\mathcal{V}}
\newcommand{\calU}{\mathcal{U}}
\newcommand{\calO}{\mathcal{O}}
\newcommand{\Int}{\mathrm{Int}}
\newcommand{\MSi}{M_{S^{(i)}}}

\newcommand{\calC}{\mathcal{C}}
\newcommand{\bhat}{\hat{\beta}}

\newcommand{\Mtot}{(U\times U)^g \times\mathcal{C}_1\times\cdots\times \mathcal{C}_l}

\newcommand{\pigl}{\pi_1(\Sigma_g\backslash\{s_1,\, ...\, ,s_l\})}
\newcommand{\prespiintro}{<\alpha_1,\beta_1,\, ...\, ,\alpha_g,\beta_g,\gamma_1,\, ...\, ,\gamma_l\ |\ \prod_{i=1}^{g} [\alpha_i,\beta_i] \prod_{j=1}^{l} \gamma_j =1>}

\newcommand{\abc}{(a_1,b_1,\, ...\, ,a_g,b_g,c_1,\, ...\, ,c_l)}
\newcommand{\relabc}{[a_1,b_1]...[a_g,b_g]c_1...c_l}
\newcommand{\relabcshort}{\prod_{i=1}^g [a_i,b_i] \prod_{j=1}^l c_j}
\newcommand{\pconj}{\mathcal{C}_1\times\cdots\times\mathcal{C}_l}
\newcommand{\HomC}{\mathrm{Hom}_{\mathcal{C}}}

\newcommand{\tU}{\widetilde{U}}
\newcommand{\tmu}{\widetilde{\mu}}

\title{A real convexity theorem for quasi-hamiltonian actions}

\author{Florent Schaffhauser}

\address{Keio University\\ 
Dept. of Mathematics\\
Hiyoshi 3-14-1\\
Kohoku-ku, 223-8522\\
Yokohama, Japon}

\email{florent@math.jussieu.fr}

\thanks{Supported by the Japanese Society for Promotion of Science (JSPS)}

\keywords{momentum maps, quasi-hamiltonian spaces, real convexity theorem}

\subjclass{53D20}

\begin{document}

\begin{abstract}
When $\qhamsp$ is a quasi-hamiltonian $U$-space with $U$ a compact connected and simply connected Lie group, the intersection of $\mu(M)$ with the exponential $\exp(\clWalc)$ of a closed Weyl alcove $\clWalc\subset\fraku=Lie(U)$ is homeomorphic, via the exponential map, to a convex polytope of $\fraku$ (\cite{AMM}). In this paper, we fix an involutive automorphism $\tau$ of $U$ such that the involution $\taum:u\mapsto\tau(u^{-1})$ leaves a maximal torus $T\subset U$ pointwise fixed (such an involutive automorphism always exists on a given compact connected Lie group $U$). We then show (theorem \ref{real_convexity_thm}) that if $\beta$ is a form-reversing involution on $M$ with non-empty fixed-point set $M^{\beta}$, compatible with the action of $(U,\tau)$ and with the momentum map $\mu$, then we have the equality $\mu(M^{\beta})\cap\exp(\clWalc)=\mu(M)\cap\exp(\clWalc)$. In particular, $\mu(M^{\beta})\cap\exp(\clWalc)$ is a convex polytope. This theorem is a quasi-hamiltonian analogue of the O'Shea-Sjamaar theorem (\cite{OSS}) when the symmetric pair $(U,\tau)$ is of maximal rank. As an application of this result, we obtain an example of lagrangian subspace in representation spaces of surfaces groups (theorem \ref{applic2}).
\end{abstract}

\maketitle

\section{Introduction}
In the following, $M$ will always denote a manifold, $U$ a compact Lie group, with Lie algebra $\fraku:=Lie(U)$, acting on $M$, and $T$ a torus with Lie algebra $\frakt:=Lie(T)$. Additionally, if $X$ is a set and $\alpha:X\to X$ is an involution on $X$, we will denote its fixed-point set $Fix(\alpha)$ by $X^{\alpha}$.\\
The convexity properties of the momentum map of a hamiltonian action have been extensively studied since the convexity theorem for hamiltonian torus actions of Atiyah (\cite{Atiyah}) and Guillemin and Sternberg (\cite{GS1}), which says that the image $\mu(M)\subset\frakt^*$ of the momentum map is a convex polyhedron. In \cite{Kirwan}, Kirwan studied the case of hamiltonian actions of non-abelian compact connected Lie groups. There, she proved a conjecture of Guillemin and Sternberg in \cite{GS2} saying that the intersection $\mu(M)\cap\clWcham$ of the image of the momentum map with any closed Weyl chamber is a convex polyhedron, usually called the \emph{momentum polytope}. In \cite{Duist}, Duistermaat studied the image $\mu(M^{\beta})$ of the fixed-point set $\Mb$ of an anti-symplectic involution $\beta:M \to M$. He then proved that for a hamiltonian torus action and under certain compatibility conditions of $\beta$ with the action of $T$ and the momentum map of this action, $\mu(\Mb)$ was also a convex polyhedron, in fact equal to the momentum polyheron $\mu(M)$. A convexity statement for the image $\mu(\Mb)$ of the fixed-point set of an anti-symplectic involution is since then often referred to as a \emph{real convexity theorem}, since in the special case where $M$ is k\"ahlerian, $M^{\beta}$ is a totally real totally geodesic lagrangian submanifold of $M$. The result of Duistermaat was then extended to hamiltonian actions of non-abelian compact connected Lie groups in \cite{OSS}. There, O'Shea and Sjamaar proved that, for an appropriate choice of a closed Weyl chamber $\clWcham$, the set $\mu(\Mb)\cap\clWcham$ is a convex subpolytope of the momentum polytope $\mu(M)\cap\clWcham$, obtained by intersecting this polytope with the fixed-point set of a certain involution $\taum$ on $\fraku^*=Lie(U)^*$. In particular, when $\taum$ fixes pointwise the closed Weyl chamber, the two polytopes are equal. In all of the above convexity theorems, the momentum map $\mu$ takes its values in the (dual of the) Lie algebra $\fraku$ and the convexity properties of such momentum maps, the momentum maps of hamiltonian actions, are well understood.\\
In \cite{AMM}, Alekseev, Malkin and Meinrenken introduced the notion of \emph{quasi-hamiltonian action}. There, the momentum map $\mu$ takes its values in the compact connected Lie group $U$ and they showed that if $U$ is in addition simply connected, the momentum map $\mu:M\to U$ enjoys convexity properties similar to those of usual momentum maps: the intersection of the image $\mu(M)$ of the momentum map with the exponential $\exp(\clWalc)$ of any closed Weyl alcove $\clWalc\subset \fraku$ is homeomorphic, via the exponential map, to a convex polyhedron of $\fraku$. Their result, whose proof is based on a convexity theorem of Meinrenken and Woodward for hamiltonian loop group actions (see \cite{MW}), is a quasi-hamiltonian analogue of the Kirwan convexity theorem. To the best of our knowledge, there does not exist a quasi-hamiltonian analogue of the O'Shea-Sjamaar theorem. Stating and proving such a real convexity theorem for quasi-hamiltonian actions is the objective and main result of this paper (see theorem \ref{real_convexity_thm}). To do this, we propose a particular construction of a symplectic slice in an arbitrary quasi-hamiltonian $U$-space $\qhamsp$. This construction is a quasi-hamiltonian analogue of the construction of Hilgert, Neeb and Plank in \cite{HNP} (based on the approach to momentum convexity of Condevaux, Dazord and Molino in \cite{CDM}). By definition, \emph{such a symplectic slice is a hamiltonian space in the usual sense}, for the action of a maximal \emph{torus} $T\subset U$. The idea of our proof is then to apply the Duistermaat theorem to this symplectic slice to obtain a real convexity theorem for quasi-hamiltonian actions. This will not work directly but can serve as a guideline for the proof, as we shall see in section \ref{slice}. Incidentally, the construction of the symplectic slice also enables us to give a proof of the Alekseev-Malkin-Meinrenken-Woodward convexity theorem for quasi-hamiltonian actions without using the corresponding statement for hamiltonian loop group actions.\\
The paper is organized as follows. In section \ref{convexity_properties}, we briefly recall the precise statements of the known convexity theorems in hamiltonian and quasi-hamiltonian geometry. In section \ref{real_convexity_properties}, we likewise recall the known real convexity theorems for hamiltonian actions in the usual sense and we state our main result, which is a real convexity theorem for quasi-hamiltonian actions (theorem \ref{real_convexity_thm}). In section \ref{slice}, we construct a particular symplectic slice in an arbitrary quasi-hamiltonian space $\qhamsp$ and give the proof of theorem \ref{real_convexity_thm}. The explicit description of this symplectic slice as the pre-image of a cell of the closed Weyl alcove $\exp(\clWalc)$ is crucial in the proof of theorem \ref{real_convexity_thm}, mostly to show that this symplectic slice contains fixed points of the involution $\beta:M\to M$. The other crucial ingredient in the proof is the local-global principle of Condevaux-Dazord-Molino and Hilgert-Neeb-Plank, which we quote from \cite{HNP}. In section \ref{applications}, the final one of this paper, we give examples of applications of the real convexity theorem for quasi-hamiltonian actions (theorem \ref{real_convexity_thm}), which include obtaining an example of lagrangian subspace in representation spaces of surface groups.

\section{Convexity properties of momentum maps}\label{convexity_properties}
\newcommand{\calE}{\mathcal{E}}
\newcommand{\calD}{\mathcal{D}}
In this section, we review the usual convexity theorems for momentum maps in the hamiltonian and quasi-hamiltonian setting. In the following, $M$ will always denote a manifold, $\w$ a $2$-form on $M$ and $U$ a compact connected Lie group acting on $M$. We will always assume that this action admits a momentum map $\mu:M\to \calE$. In this paper, we will consider two cases: $\calE=\fraku^*$, the dual of the Lie algebra $\fraku:=Lie(U)$ (hamiltonian case),  and $\calE=U$ (quasi-hamiltonian case). The target space $\calE$ of the momentum map is also endowed with an action of $U$ and the momentum map $\mu$ is assumed to be equivariant (if $\calE=\fraku$, the action of $U$ on $\calE$ is the co-adjoint action, if $\calE=U$ it is the conjugacy action).

\subsection{Stating a convexity theorem}\label{meaning}

Stating a convexity theorem for the momentum map $\mu:M\to \calE$ consists in showing that the image $\mu(M)$ of the momentum map has certain convexity properties: by intersecting it with a fundamental domain $\calD$ for the action of $U$ on $\calE$, the resulting set $\mu(M)\cap\calD$ is homeomorphic to a convex subset of a vector space. In this paper, this vector space will always be finite-dimensional. In fact, it will be the Lie algebra $\fraku=Lie(U)$. There are nonetheless known examples of convexity theorems for momentum maps in inifinite dimension (see for instance \cite{MW}). Observe that, since $\mu:M\to\calE$ is $U$-equivariant, $\mu(M)$ is a union of $U$-orbits. Therefore, considering the intersection $\mu(M)\cap\calD$ is equivalent to considering the projection of $\mu(M)$ onto the orbit space $\calE/U\simeq \calD$: the projection $p:\calE\to\calE/U$ induces a homeomorphism between $\mu(M)\cap\calD$ and $p\circ\mu(M)$ (recall that a fundamental domain $\calD$ contains exactly one point of each orbit). With this approach we see that, for the notion of convexity to make sense, it is sufficient to be able to identify the orbit space $\calE/U$ with a convex subset of a vector space (see in particular definition \ref{partie_convexe}). We now give the classical examples of convexity theorems for momentum maps.

\subsection{The Atiyah-Guillemin-Sternberg theorem}\label{AGS}

The following theorem, proved independently by Atiyah in \cite{Atiyah} and by Guillemin and Sternberg in \cite{GS1}, is historically the first convexity result for momentum maps. It says that the image $\mu(M)$ of the momentum map $\mu:M\to \frakt^*$ of a hamiltonian \emph{torus} action is a convex polyhedron. In this case, $U=T$ is an \emph{abelian} compact connected Lie group and $\calE=\frakt^*$ is the dual of the Lie algebra $\frakt=Lie(T)$. The co-adjoint action is trivial and so the fundamental domain is $\calD=\frakt^*$. The original statement of the Atiyah-Guillemin-Sternberg (AGS) theorem is for $M$ compact, in which case $\mu(M)$ is a compact convex poyhedron (a convex polytope), whose vertices are the images under $\mu$ of fixed points of the action of $T$ on $M$. This statement was later generalized to non-compact manifolds by Hilgert, Neeb and Plank in \cite{HNP}, based on the work of Condevaux, Dazord and Molino in \cite{CDM} (see also \cite{Sjamaar}, \cite{LMTW} and \cite{Benoist}). This generalization is obtained under the assumption that the momentum map is proper (in particular, closed).
\begin{thm}[Momentum convexity for hamiltonian torus actions]\cite{Atiyah,GS1}\label{AGS_thm}
Let $(M,\w)$ be a connected symplectic manifold endowed with a hamiltonian action of a torus $T$ whose momentum map $\mu:M\to\frakt^*$ is proper. Then:
\begin{enumerate}
\item[(i)] $\mu(M)$ is a closed polyhedral convex set.
\item[(ii)] the fibre $\mu^{-1}(\{v\})$ of $\mu$ above any $v\in\frakt^*$ is a connected set.
\item[(iii)] $\mu:M\to\mu(M)$ is an open map.
\item[(iv)] if $\mu(x)\in \mu(M)$ is an extremal point of the convex set $\mu(M)$, then $x\in M$ is a fixed point of the action: for all $t\in T$, $t.x=x$.
\end{enumerate}
If $M$ is compact then $\mu(M)$ is a convex polytope and it is the convex hull of the images under $\mu$ of the fixed points of the action.
\end{thm}\newcommand{\calH}{\mathcal{H}}
This theorem enjoys very nice applications, some of which in fact motivated the search for such a statement. For instance, the conjugacy action of the torus of diagonal unitary matrices $T\subset U(n)$ on a conjugacy class of hermitian matrices $\calO\subset\calH(n)$ (a co-adjoint orbit of $U(n)$) admits as a momentum map the map taking a hermitian matrix $H$ to the collection $diag(H)\in\frakt^*=diag(\calH(n))$ of its diagonal entries. The AGS theorem then implies the Schur-Horn theorem: the diagonal of a hermitian matrix is a convex combination of permutations of its eigenvalues: $$diag(H)=\sum_{\s\in\mathfrak{S}_n}a_{\s}(\lambda_{\s_1},\, ...\, ,\lambda_{\s_n})$$ with $(\lambda_1,\, ...\, ,\lambda_n)=Spec(H)\subset\R^n\simeq Lie(T)$ and $\sum_{\s\in\mathfrak{S}_n}a_{\s}=1, a_{\s}\geq 0$. The reason why it is important to have such a statement for non-compact $M$ is that it will later be applied to \emph{symplectic slices} (also called \emph{symplectic cross-sections}) to prove more general convexity results by reduction to the case of a torus action. Symplectic slices in a $U$-space $(M,\w,\mu:M\to \calE)$ are symplectic manifolds $N\subset M$ such that $\overline{U.N}=M$ and endowed with a hamiltonian action of a maximal torus $T\subset U$. They also satisfy $\overline{\mu(N)}=\mu(M)\cap\calD$, where $\calD$ is a fundamental domain for the $U$-action on $\calE$. These symplectic slices are non-compact manifolds. All of this will be developped in section \ref{slice} in the case of quasi-hamiltonian spaces $\qhamsp$. We now move on to hamiltonian $U$-spaces $(M,\w,\mu:M\to\fraku^*)$ with $U$ non-abelian.

\subsection{The Kirwan-Guillemin-Sternberg theorem}\label{KGS}

The following theorem, proved by Kirwan in \cite{Kirwan} (and by Guillemin and Sternberg in \cite{GS2} in the case where $M$ is a K\"ahler manifold), deals with hamiltonian actions of compact connected groups $U$. In the notation of subsection \ref{meaning}, $\calE=\fraku^*$ is the dual of the Lie algebra $\fraku=Lie(U)$. A fundamental domain for the co-adjoint action of $U$ on $\fraku^*$ is homeomorphic to a closed Weyl chamber in an arbitrary Cartan subalgebra $\frakt\subset\fraku$, and is usually denoted by $\overline{\frakt^*_+}$. In particular, $\fraku^*/Ad(U)\simeq\clWcham$ is homeomorphic to a closed convex subset of a finite-dimensional vector space. As in subsection \ref{AGS}, the original statement of the Kirwan-Guillemin-Sternberg (KGS) theorem was for compact $M$ and was generalized to non-compact $M$ in \cite{HNP} (see also \cite{Sjamaar} and \cite{Benoist}) provided the momentum map is proper (in particular, closed). To state the theorem, it is convenient to assume that there is a given $Ad$-invariant non-degenerate bilinear form on $\fraku=Lie(U)$, so that we can identify $\fraku^*$ and $\fraku$ equivariantly and, for any subalgebra $\mathfrak{h}\subset\fraku$, think of $\mathfrak{h}$ as a subspace of $\fraku^*$. In particular, a Weyl chamber may be thought of as a subset of $\fraku^*$.
\begin{thm}[Momentum convexity for hamiltonian actions of compact groups]\cite{Kirwan}\label{KGS_thm}
Let $(M,\w)$ be a connected symplectic manifold endowed with a hamiltonian action of a compact connected Lie group $U$ whose momentum map $\mu:M\to \fraku^*$ is proper. Then, for any choice of a Cartan subalgebra $\frakt\subset u$ and any choice of a closed Weyl chamber $\clWcham \subset\frakt^* \subset \fraku^*$, the set $\mu(M)\cap\clWcham$ is a closed convex polyhedron of $\frakt^*$. Furthermore, the map $\mu:M\to \fraku^*$ has connected fibres and if $\calU$ is open in $M$ then $\mu(\calU)\cap\clWcham$ is open in $\mu(M)\cap\clWcham$.
\end{thm}
In \cite{Sjamaar}, Sjamaar showed that if additionally $M$ is compact, then $\mu(M)\cap\clWcham$ is the convex hull of the set $\mu(E)$ where $E=\{x\in M\ |\ \mu(x)\in\clWcham$ and $\fraku_{\mu(x)}=[\fraku_{\mu(x)},\fraku_{\mu(x)}]\oplus\fraku_x\}$ and that this set $\mu(E)$ is a discrete subset of $\clWcham$ (this last statement being true even if $M$ is not compact). There are now many proofs of the KGS convexity theorem, for which we refer to the above-cited papers. The one we will often use as a guiding analogue in our study of the quasi-hamiltonian case is the one given by Hilgert, Neeb and Plank in \cite{HNP}. The general argument, based on ideas of Guillemin and Sternberg in \cite{GS1} and Condevaux, Dazord and Molino in \cite{CDM}, is as follows: one constructs a symplectic slice $N\subset M$ to which one applies the AGS theorem to show that $\mu(N)$ is convex; since by construction $\overline{\mu(N)}=\mu(M)\cap\clWcham$, one gets convexity of $\mu(M)\cap\clWcham$. The main difficulty is to be able to apply the AGS theorem to the symplectic slice $N$. As a matter of fact, the AGS theorem does \emph{not} apply directly to $N$ (even the version of the AGS theorem for non-compact manifolds). But the use of the local-global principle enables one to overcome this difficulty (see \cite{HNP}). We will see in subsection \ref{proof_AMMW} how one can proceed in the same way in the quasi-hamiltonian case.\\ As an application of theorem \ref{KGS_thm}, consider the case where $U=U(n)$ and $M=\calO_{\lambda_1}\times\calO_{\lambda_2}$ is a product of co-adjoint orbits. Then, for $(H_1,H_2)\in\calO_{\lambda_1}\times\calO_{\lambda_2}$, one has $\mu(H_1,H_2)=H_1+H_2$ and $\mu(M)\cap\clWcham$ consists of the possible (ordered) spectra for the sum of two hermitian matrices. By theorem \ref{KGS_thm}, these possible spectra form a closed convex subset of $\R^n\simeq Lie(T)$ (this is the Weyl-Horn problem; see \cite{Ber-Sj} for a symplectic approach to this problem).

\subsection{Momentum maps with values in a compact connected Lie group}\label{AMMW}

We now turn to momentum maps with values in a compact connected Lie group. More precisely, the notion we shall be dealing with in the rest of this paper is the notion of \emph{quasi-hamiltonian space}, which is due to Alekseev, Malkin and Meinrenken. In the notation of subsection \ref{meaning}, the momentum map $\mu:M\to\calE$ takes its values in the Lie group $U$ acting on the manifold $M$: one has $\calE=U$ and a fundamental domain $\calD$ for the conjugacy action of $U$ on itself is homeomorphic to the quotient of the closed Weyl alcove $\clWalc\subset \frakt=Lie(T)$ by the affine action of $\pi_1(U)$ on $\frakt$, where $T$ is a maximal torus of $U$ (see for instance \cite{Bourbaki}, p.45). In particular, $U/Int(U) \simeq \clWalc / \pi_1(U)$ is homeomorphic to the closed convex polyhedron $\clWalc\subset \frakt$ if and only if $\pi_1(U)=0$. Thus, when $U$ is simply connected, the notion of a convex subset of $U/Int(U)$ makes sense:
\begin{defi}\label{partie_convexe}
Let $U$ be a compact connected simply connected Lie group. A subset $A\subset U/Int(U)$ is called \emph{convex} if it is mapped, under the homeomorphism $U/Int(U)\simeq \clWalc\subset \frakt$, to a convex subset of $\frakt$.
\end{defi}
Observe additionally that $\clWalc$ is a \emph{compact} polyhedron. When $U$ is not simply connected, $\clWalc/ \pi_1(U)$ is not simply connected either, and therefore cannot be homeomorphic to a convex subset of a vector space. From now on, we will always assume that the compact connected Lie group $U$ is simply connected. Then, there is a convenient way of understanding the identification $U/Int(U)\simeq \clWalc$: the exponential map $\exp:\fraku=Lie(U)\to U$ restricts to an homeomorphism $\exp|_{\clWalc}:\clWalc\to\exp(\clWalc)$ and $\calD:=\exp(\clWalc)$ is a fundamental domain for the conjugacy action of $U$ on itself. A subset $A\subset \exp(\clWalc)$ is then called convex if $\exp^{-1}(A)\subset \frakt$ is convex. We then have the following convexity statememt, which we quote from \cite{AMM} (where it is deduced from the convexity theorem for hamiltonian loop group actions of \cite{MW}):
\begin{thm}[Momentum convexity for quasi-hamiltonian actions of compact groups]\cite{AMM}\label{AMMW_thm}
Let $U$ be a compact connected simply connected Lie group and let $\qhamsp$ be a connected quasi-hamiltonian $U$-space with proper momentum map $\mu:M\to U$. Then, for any choice of a maximal torus $T\subset U$ and any choice of a closed Weyl alcove $\clWalc\subset\frakt=Lie(T)$, the set $(\mu(M)\cap\exp(\clWalc))\subset \exp(\clWalc)$ is a closed convex polyhedron of $\exp(\clWalc)\simeq \clWalc\subset\frakt$. Furthermore, the map $\mu:M\to U$ has connected fibres and if $\calU$ is open in $M$ then $\mu(\calU)\cap\exp(\clWalc)$ is open in $\mu(M)\cap\exp(\clWalc)$.
\end{thm}
\begin{rema}[On the assumption of properness of the momentum map]
When $U$ is a compact Lie group, the known examples of quasi-hamiltonian $U$-spaces \emph{(}\cite{AMM}, \cite{HJS}\emph{)} are compact, so the assumption that $\mu$ be a proper map is superfluous. Nonetheless, the convexity theorem remains true under this assumption, so that it could be applied to non-compact quasi-hamiltonian spaces, should such examples be discovered.
\end{rema}
As explained in \cite{AMM}, although this convexity statement concerns quasi-hamiltonian actions, it can be deduced from the analoguous statement for hamiltonian loop group actions. In subsection \ref{proof_AMMW}, we will give another proof of this theorem, based on the approach to momentum convexity for usual hamiltonian actions developped in \cite{HNP}: we construct a particular symplectic slice $N\subset M$ for any quasi-hamiltonian action and we use the local-global pinciple of \cite{HNP} to show that $\mu(N)$ is convex. In subsection \ref{main_proof}, we use this construction of a symplectic slice to give a proof of our main result (theorem \ref{real_convexity_thm}). Before ending this section, we give an example of application of theorem \ref{AMMW_thm}, which is an analogue of the application of the KGS theorem given in subsection \ref{KGS}. Consider the case where $U=U(n)$ and $M=\calC_{\lambda_1}\times\calC_{\lambda_2}$ is a product of conjugacy classes. Then, for $(u_1,u_2)\in\calC_{\lambda_1}\times\calC_{\lambda_2}$, one has $\mu(u_1,u_2)=u_1u_2$ and $\mu(M)\cap\exp(\clWalc)$ consists of the possible (ordered) spectra for the product of two unitary matrices. By theorem \ref{AMMW_thm}, these possible spectra form a closed convex subset of $\R^n\simeq Lie(T)$.

\section{Real convexity theorems}\label{real_convexity_properties}

The notion of \emph{real convexity theorem} appears when one considers not the full image $\mu(M)$ of the momentum map but rather the image $\mu(M^{\beta})$ of the fixed-point set of an involution $\beta:M\to M$ satisfying $\beta^*\w=-\w$. In the usual hamiltonian case, $\w$ is a symplectic form and $M^{\beta}$ is a lagrangian submanifold of $M$. In particular when $M$ is a K\"ahler manifold, $M^{\beta}$ is a \emph{totally real} totally geodesic submanifold of $M$. This inspired the terminology \emph{real convexity theorem} for convexity properties of $\mu(M^{\beta})$. 

\subsection{Stating a real convexity theorem}\label{meaning_real}

Recall from subsection \ref{meaning} that convexity theorems for momentum maps $\mu:M\to\calE$ are obtained by intersecting the image of $\mu(M)$ with a fundamental domain domain $\calD$ of the action of $U$ on $\calE$. Since the work of O'Shea and Sjamaar in \cite{OSS}, it is known that to obtain a real convexity theorem for hamiltonian actions of compact groups, one needs to assume that the Lie group $U$ is endowed with an involutive automorphism $\tau$ and that the anti-symplectic involution $\beta$ on $M$ is compatible with the action of $U$ and with the momentum map $\mu:M\to \fraku^*$ in the sense that $\beta(u.x)=\tau(u).\beta(x)$ and $\mu\circ\beta=\taum\circ\mu$, where $\taum$ denotes the involution $(-d\tau)^*:\fraku^*\to\fraku^*$. A real convexity theorem is then a statement on the convexity properties of $\mu(M^{\beta})$, namely that $\mu(M^{\beta})\cap\calD$ is a convex polyhedron and in fact a subpolyhedron of the momentum polyhedron $\mu(M)\cap\calD$. More precisely, for an appropriate choice of $\calD$, one has $\mu(M^{\beta})\cap\calD=(\mu(M)\cap\calD)\cap Fix(\taum)$. Observe that by compatibility of $\beta$ with the momentum map $\mu:M\to\fraku^*$, one automatically has $\mu(M^{\beta})\subset Fix(\taum)$. In this hamiltonian context, choosing $\calD$ appropriately means choosing the closed Weyl chamber $\clWcham$ such that $\clWcham\cap Fix(\taum)$ is a fundamental domain for the action of (the neutral component of) $U^{\tau}=Fix(\tau)\subset U$ on $Fix(\taum)\subset \fraku^*=\calE$. Based on this remark, a general real convexity theorem should be a statement of the following form: the intersection of the image $\mu(M)$ of the momentum map with a fundamental domain $\calD_0$ of the action of $U^{\tau}$ on $Fix(\taum)$ is a convex set, equal to $\mu(M^{\beta})\cap\calD_0$. This will indeed be the case in theorem \ref{real_convexity_thm}. There, $\calE=U$ and $\taum(u)=\tau(u^{-1})$. The fixed-point set $Fix(\taum)$ is not connected in general but the neutral component $K$ of $U^{\tau}$ acts on the connected component $Q_0$ of $1$ in $Fix(\taum)$. If the closed Weyl alcove $\clWalc$ is chosen appropriately, then $\exp(\clWalc)\cap Q_0$ is a fundamental domain for this action and it is natural to conjecture that if $\beta:M\to M$ is a form-reversing compatible involution on the quasi-hamiltonian space $M$, then $\mu(M^{\beta})\cap\exp(\clWalc)$ is a convex subpolyhedron of $\mu(M)\cap\exp(\clWalc)$, obtained by intersecting it with $Q_0$. In theorem \ref{real_convexity_thm}, we prove this conjecture under the assumption that the symmetric pair $(U,\tau)$ is of maximal rank, meaning that there exists a maximal torus $T\subset U$ fixed \emph{pointwise} by $\taum$ (such an involution always exists on an arbitrary compact connected Lie group, see for instance \cite{Loos}). In this case, the fundamental domain $\exp(\clWalc)\cap Q_0$ is simply $\exp(\clWalc)\subset T\subset Q_0$ (for instance: if $\tau(u)=\overline{u}$ on $U=U(n)$, every symmetric unitary matrix is conjugate to a diagonal one by a real orthogonal matrix).

\subsection{The Duistermaat theorem}\label{Duist}

The Duistermaat theorem (\cite{Duist}) is a real version of the AGS theorem, the involutive automorphism on the (abelian) group $T$ being $\tau(t)=t^{-1}$. In particular, $\taum=Id_{\frakt^*}$. The original statement was for $M$ compact and was extended to non-compact $M$ in \cite{HNP}.
\begin{thm}[A real convexity theorem for hamiltonian torus actions]\cite{Duist}\label{Duist_thm}
Let $(M,\w,\mu:M\to\frakt^*)$ be a connected hamiltonian $T$-space with proper momentum map. Let $\beta:M\to M$ be an involution satisfying:
\begin{enumerate}
\item[(i)] $\beta^*\w=-\w$.
\item[(ii)] $\beta(t.x)=t^{-1}.\beta(x)$ for all $x\in M$ all $t\in T$.
\item[(iii)] $\mu\circ\beta=\mu$.
\item[(iv)] $M^{\beta}\not=\emptyset$.
\end{enumerate}
Then, for every connected component $L\subset M^{\beta}$ of the fixed-point set of $\beta$, $\mu(L)$ is a closed convex set and one has: $\mu(L)=\mu(M)$.
\end{thm}
We will use some of the ideas of the proof of Duistermaat's theorem given in \cite{HNP} in subsection \ref{main_proof}. As an application of theorem \ref{Duist_thm}, let us go back to the example given in subsection \ref{AGS}. In this case, the Duistermaat theorem says: the diagonal of a real symmetric matrix is a convex combination of permutation of its eigenvalues \emph{and} any such convex combination is the diagonal of a real symmetric matrix with the same spectrum.

\subsection{The O'Shea-Sjamaar theorem}\label{OSS}

The O'Shea-Sjamaar theorem is a real version of the KGS theorem. The proof given in \cite{OSS} consists, as in \cite{Sjamaar}, in establishing the result for projective and affine varieties and then handle the general case using a local normal form theorem.
\begin{thm}[A real convexity theorem for hamiltonian actions of compact groups]\cite{OSS}\label{OSS_thm}
Let $(M,\w,mu:M\to\fraku^*)$ be a connected hamiltonian $U$-space with proper momentum map. Let $\tau:U\to U$ be an involutive automorphism of $U$. Denote by $\taum$ the involution $(-d\tau)^*:\fraku^*\to\fraku^*$ and by $\mathfrak{q}$ the $(-1)$-eigenspace of $d\tau:\fraku\to\fraku$. Choose a maximal abelian subspace $\mathfrak{a}$ of $\mathfrak{q}$, and a Cartan subalgebra $\frakt$ of $\fraku$ containing $\mathfrak{a}$ \emph{(}such a Cartan subalgebra is globally stable under $d\tau$\emph{)} and a closed Weyl chamber $\clWcham\subset\fraku^*$. Let $\beta:M\to M$ be an involution satisfying:
\begin{enumerate}
\item[(i)] $\beta^*\w=-\w$.
\item[(ii)] $\beta(u.x)=\tau(u).\beta(x)$ for all $u\in U$ and all $x\in M$.
\item[(iii)] $\mu\circ \beta=\taum\circ\mu$.
\item[(iv)] $M^{\beta}\not=\emptyset$.
\end{enumerate}
Then, for every connected component $L\subset M^{\beta}$, $\mu(L)\cap\clWcham$ is a closed convex polyhedron obtained by intersecting the momentum polyhedron $\mu(M)\cap\clWcham$ with $Fix(\taum)\subset \frakt^*$: $$\mu(L)\cap\clWcham=(\mu(M)\cap\clWcham)\cap Fix(\taum).$$
\end{thm}
As an application of this theorem, let us come back to the example given in subsection \ref{KGS}. When $U=U(n)$ and $\tau:u\mapsto\overline{u}$, the Cartan subalgebra consisting of diagonal hermitian matrices is fixed pointwise by $\taum:H\mapsto H^t$ (the symmetric pair $(U,\tau)$ is of maximal rank): one has $\mathfrak{a}=\frakt$ and therefore $\clWcham\subset Fix(\taum)$, hence $\mu(L)\cap\clWcham=\mu(M)\cap\clWcham$. Consequently, the possible spectra of a sum $S_1+S_2$ of two real symmetric matrices are the same as those of a sum $H_1+H_2$ of two hermitian matrices satisfying $Spec(H_i)=Spec(S_i)$ for $i=1,2$.

\subsection{A real convexity theorem for quasi-hamiltonian actions}\label{main_result}

We can now state the main result of this paper, which is a real convexity theorem for quasi-hamiltonian actions. This result is a quasi-hamiltonian analogue of the O'Shea-Sjamaar theorem in the special case where the symmetric pair $(U,\tau)$ is of maximal rank.

\begin{thm}[A real convexity theorem for quasi-hamiltonian actions of compact groups]\label{real_convexity_thm}
Let $U$ be a compact connected simply connected Lie group and let $\qhamsp$ be a quasi-hamiltonian $U$-space with proper momentum map $\mu:M\to U$. Let $\tau:U\to U$ be an involutive automorphism of $U$ such that there exists a maximal torus $T\subset U$ fixed pointwise by the involution $\taum:u\in U\mapsto \tau(u^{-1})$ and let $\clWalc\subset \frakt=Lie(T)$ be a closed Weyl alcove. Let $\beta:M\to M$ be an involution satisfying:
\begin{enumerate}
\item[(i)] $\beta^*\w=-\w$.
\item[(ii)] $\beta(u.x)=\tau(u).\beta(x)$ for all $u\in U$ and all $x\in M$.
\item[(iii)] $\mu\circ \beta=\taum\circ\mu$.
\item[(iv)] $M^{\beta}\not=\emptyset$.
\item[(v)] $\mu(M^{\beta})$ intersects the connected component $Q_0$ of $1$ in $Fix(\taum)$.
\end{enumerate}
Then, for every connected component $L$ of $M^{\beta}$ satisying $\mu(L)\cap Q_0\not=\emptyset$, one has: $$\mu(L)\cap\exp(\clWalc)=\mu(M)\cap\exp(\clWalc).$$ In particular, $\mu(L)\cap\exp(\clWalc)$ is a convex polytope of $\exp(\clWalc)\simeq \clWalc\subset\frakt$, equal to the full momentum polytope $\mu(M)\cap\exp(\Walc)$.
\end{thm}
Observe that $\exp(\clWalc)\subset T\subset Fix(\taum)$ and contains $1\in U$. Since $\exp(\clWalc)$ is connected, this implies $\exp(\clWalc)\subset Q_0$. Assumption (v) is of topological nature and accounts for the fact that, in a Lie group, the subset $Fix(\taum)$ is not necessarily connected. When $Fix(\taum)$ is connected (for instance if $U=SU(n)$ and $\tau(u)=\overline{u}$, since in this case every symmetric unitary matrix is of the form $\exp(iS)$ with $S$ a real symmetric matrix), this assumption is automatically satisfied (recall that by compatibility of the $\beta$ with the momentum map one has $\mu(M^{\beta})\subset Fix(\taum)$). In the O'Shea-Sjamaar theorem there is no such assumption because $Fix(\taum)$ is a vector subspace of $\fraku^*$ and is therefore connected. The rest of this paper is devoted to the proof of theorem \ref{real_convexity_thm} and to some applications.\\ Observe that here, in contrast with subsection \ref{OSS}, we do not have an application of theorem \ref{real_convexity_thm} to the spectrum of a product of real orthogonal matrices (which would be a real version of the application of theorem \ref{AMMW_thm} given in subsection \ref{AMMW}). The reason for this is that if $\tau:u\mapsto \overline{u}$ on $U=U(n)$ and $M=\calC_1\times\calC_2$ is a product of two conjugacy classes, if we set $\beta(u_1,u_2)=(\overline{u_1},\overline{u_2})$ we do not have $\mu\circ\beta=\taum\circ\mu$ (for lack of commutativity in $U$). Instead we refer to theorem \ref{applic3} for an application of theorem \ref{real_convexity_thm} that would be a real version of the statement on the possible spectrum of a product of two unitary matrices.

\section{Constructing a symplectic slice in a quasi-hamiltonian space}\label{slice}

In this section, we prove the real convexity result for group-valued momentum maps stated in theorem \ref{real_convexity_thm}. The idea of the proof is to apply the Duistermaat theorem to an appropriate symplectic slice $N\subset M$. The point that requires the most attention is showing that $N^{\beta}\not=\emptyset$. The construction of the symplectic slice itself relies on adapting the ideas of \cite{HNP} to the quasi-hamiltonian setting. In particular, using this approach, we will be able, in subsection \ref{proof_AMMW}, to give a proof of the Alekseev-Malkin-Meinrenken-Woodward convexity theorem.

\subsection{The local-global principle}\label{LKP}

We recall here the results of \cite{HNP} that we will later apply to the symplectic slice $N\subset M$. The following theorem is obtained by using the local normal form theorem for hamiltonian torus actions.

\begin{thm}[Local convexity theorem for hamiltonian torus actions]\cite{HNP}\label{local_convexity_thm}
Let $(N,\w)$ be a symplectic manifold endowed with a hamiltonian action of a
torus $T$ with momentum map $\mu:N\to\mathfrak{t}^*$. Then for every $x\in
N$, there exist an open neighbourhood $\calV_x$ of $x\in N$ and a polyhedral
cone $C_{\mu(x)}\subset\mathfrak{t}^*$ with vertex $\mu(x)$ such that:
\begin{enumerate}
\item[(i)] $\mu: \calV_x \to C_{\mu(x)}$ is an open map. In particular,
$\mu(\calV_x)$ is an open neighbourhood of $\mu(x)$ in $C_{\mu(x)}$.
\item[(ii)] $\mu^{-1}(\{\mu(y)\})\cap\calV_x$ is connected for all $y\in\calV_x$.
\end{enumerate}
If in addition $\beta$ is an antisymplectic involution on $N$
satisfying $\beta(t.x)=t^{-1}.\beta(x)$ and $\mu\circ\beta=\mu$, then
assertion (i) above remains true for the manifold
$N^{\beta}:=Fix(\beta)$ and the \emph{same} cones $C_{\mu(x)}$, for $x\in
N^{\beta}$, that is:
\begin{enumerate}
\item[(iii)] $\mu: \calV_x \cap N^{\beta} \to C_{\mu(x)}$ is an
open map. In particular, $\mu(\calV_x \cap N^{\beta})$ is an open
neighbourhood of $\mu(x)$ in $C_{\mu(x)}$.
\end{enumerate}
\end{thm}
For additional local properties, including a description of the
cones $C_{\mu(x)}$ using the local normal form of the
action, we refer to \cite{HNP}. Conditions (ii) and (iii) above play a special role when it comes
to convexity considerations insofar as they make it possible to obtain a global
result from a local one (see theorem \ref{LKP_thm}), which justifies the
following definition:
\begin{defi}[Local convexity data]\label{localconvexitydata}
Let $X$ be a \emph{connected} Hausdorff space, and let $V$ be a finite
dimensional vector space. Consider a continuous map $\psi:X\to V$. It is said
that $\psi$ \emph{gives rise to local convexity data}
$(\calV_x,C_{\psi(x)})_{x\in X}$ if for any $x\in X$ there exist an open
neighbourhood $\calV_x$ of $x$ in $X$ and a convex cone $C_{\psi(x)}\subset
V$ with vertex $\psi(x)$ such that:
\begin{enumerate}
\item[(O)] $\psi:\calV_x\to C_{\psi(x)}$ is an open map.
\item[(LC)] $\psi^{-1}(\{\psi(y)\})\cap\calV_x$ is connected for all $y\in
\calV_x$.
\end{enumerate}
A map $\psi:X\to V$ satisying condition (LC) alone is said to be
\emph{locally fibre-connected}.
\end{defi}
As an example, theorem \ref{local_convexity_thm} shows that the momentum map of a hamiltonian torus action gives rise to local convexity data. We then have:
\begin{thm}[Local-global principle]\emph{\cite{HNP}}\label{LKP_thm}
Let $\psi:X\to V$ be a map giving rise to local convexity data
$(\calV_x,C_{\mu(x)})_{x\in X}$ and assume that $\psi$ is a proper map. Then
$\psi(X)$ is a closed locally polyhedral convex subset of $V$, the fibres
$\psi^{-1}(\{v\})$ are connected for all $v\in V$, $\psi:X\to \psi(X)$ is an
open map and $C_{\mu(x)}=\psi(x) + \overline{\R^+.(\psi(X)\bs\{\psi(x)\})}$.
\end{thm}
In particular, we see that theorem \ref{AGS_thm} immediately follows
from theorems \ref{local_convexity_thm} and \ref{LKP_thm}. In order to obtain a Duistermaat-like statement for our symplectic slice $N\subset M$, we will need the following two results, the first of which is a corollary of the local-global principle (these two results will be used in the proof of proposition \ref{Duist-like}).
\begin{cor}\emph{\cite{HNP}}\label{real_Prinzip}
Let $V$ be a finite dimensional vector space and let $P\subset V$ be a
closed connected subset of $V$ such that for all $x\in P$, there exists a
neighbourhood $\calO_v$ of $v$ in $V$ and a cone $C_v\subset V$ with vertex
$v$ such that $\calO_v\cap P=O_v\cap C_v$. Then $P$ is a convex subset of
$V$ and for all $v\in V$, $C_v=v+\overline{\R^+.(P\bs\{v\})}$. 
\end{cor}
\begin{lem}\emph{\cite{HNP}}\label{inclusion_de_deux_convexes}
If $P_1\subset P_2$ is an inclusion between two convex subsets of a
finite-dimensional vector space satisfying, for all $v\in P_1$ the condition
$v+\overline{\R^+.(P_1\bs\{v\})}=v+\overline{\R^+.(P_2\bs\{v\})}$, then
$P_1=P_2$.
\end{lem}
\begin{proof}
A convex set is the intersection of cones containing it, so that:
$$P_1=\bigcap_{v\in P_1} \big(v+\overline{\R^+.(P_1\bs\{v\})}\big) =
\bigcap_{v\in P_1} \big(v+\overline{\R^+.(P_2\bs\{v\})}\big) \supset
\bigcap_{v\in P_2} \big(v+\overline{\R^+.(P_2\bs\{v\})}\big) = P_2$$ hence
$P_1=P_2$.
\end{proof}
There is one more difficulty one is faced with when trying to apply theorem \ref{AGS_thm} or \ref{Duist_thm} to a symplectic slice. As mentioned in subsection \ref{AGS}, symplectic slices are non-compact manifolds (which in fact motivated the search for a strengthened version of the Atiyah-Guillemin-Sternberg and Duistermaat theorems in \cite{CDM,HNP}) and in addition to that, the momentum map $\mu|_N$ is not a proper map. In subsection \ref{proof_AMMW}, we will see how to circumvent this difficulty by applying the following result:
\begin{prop}\emph{\cite{HNP}}\label{ascending_sequence}
Let $\psi:X\to V$ be a map giving rise to local convexity data
$(\calV_x,C_{\psi(x)})_{x\in X}$. Consider any closed locally polyhedral
convex subset $D\subset V$ and set $Y:=\psi^{-1}(D)\subset X$. Then
$\psi|_Y:Y\to V$ gives rise to local convexity data
$(\calV_y,C_{\psi(y)}\cap\overline{\R^+.(D \bs \{\psi(y)\}) })_{y\in Y}$.
\end{prop}
Finally, we quote the following result from \cite{Benoist}. It is a purely topological argument that we will use in subsection \ref{proof_AMMW}. We refer to \cite{Benoist} for a proof and comments on the fact that the assumption that $f(\Omega)$ be convex is necessary.
\begin{lem}\cite{Benoist}\label{connexite_des_fibres}
Let $X$ be a metric space and $\Omega$ be an open dense subset of $M$. Let $f:X\to \R^d$ be a proper continuous map. Assume that the image $f(\Omega)$ is a convex subset of $\R^d$ and that for all $x\in \Omega$, the fibre $f^{-1}(\{f(x)\})$ is connected. Then all the fibres of $f$ are connected.
\end{lem}

\subsection{Construction of a symplectic slice}\label{slice_construction}

The purpose of this subsection is to prove the following result:
\begin{thm}[Existence of a symplectic slice]\label{symp_slice}
Let $U$ be a compact connected simply connected Lie group and let
$(M,\w,\mu:M\to U)$ be a connected quasi-hamiltonian $U$-space with proper momentum
map $\mu$. Let $T\subset U$ be a maximal torus in $U$, let
$\overline{\mathcal{W}}\subset\mathfrak{t}=Lie(T)$ be the closure of a Weyl
alcove and let $p:U\to U/\Int(U)$ be the projection from $U$ to the set
of its conjugacy classes. Recall that the exponential map $\exp : \mathfrak{t}
\to T$ induces a
homeomorphism $\overline{\mathcal{W}}\overset{\simeq}{\longto} U/\Int(U)$.
Then, there exists a submanifold $N\subset M$ such that:
\begin{enumerate}
\item[(i)] $N$ is connected.
\item[(ii)]$N$ is $T$-stable.
\item[(iii)] $\w|_N$ is a symplectic form.
\item[(iv)] the action of $T$ on $N$ is hamiltonian with momentum map the
map $$\widetilde{\mu}:=p\circ\mu|_N : N \longto U/\Int(U)
\simeq \overline{\mathcal{W}} \subset \mathfrak{t}.$$
\item[(v)] The set $U.N:=\{u.x\, :\,x\in N, u\in U\}$ is dense in $M$, and
the set $\widetilde{\mu}(N)$ is dense in $\tilde{\mu}(M)$.
\end{enumerate}
\end{thm}
The manifold $N\subset M$ whose existence is guaranteed by theorem \ref{symp_slice} is called a \emph{symplectic slice} because it is a
symplectic manifold satisfying $\overline{U.N}=M$. Theorem \ref{symp_slice} is a quasi-hamiltonian analogue of a result of Hilgert, Neeb and Plank in \cite{HNP}. The details of the construction, that we now set out to give, will show some remarkable similarities with the usual hamiltonian setting.

\subsubsection{\textbf{Structure of the Weyl alcove}}\label{structure_Walc}

We begin by giving a description of the closed Weyl alcove $\clWalc\subset \frakt=Lie(T)$ in terms of the root system of $(U,T)$. We refer to \cite{Loos} and \cite{Bourbaki} for a proof of the classical results on root systems that we shall use. The compact connected simply connected Lie group $U$ is endowed with an $Ad$-invariant positive definite product $(.\, |\, .)$ (for instance, minus the Killing form). In particular, we may identify equivariantly $\fraku^*$ with $\fraku$ and $\frakt^*$ with $\frakt$. Consider now the adjoint action of the maximal torus $T\subset
U$ on the complexification $\mathfrak{u}^{\C}$ of $\mathfrak{u}=Lie(U)$, and
denote by $R$ the corresponding root system: $$R:=\big\{\alpha\in \mathfrak{t}^*\ |\
\mathfrak{u}^{\C}_{\alpha}\not=\{0\} \big\}$$ where for any linear form
$\alpha:\mathfrak{t}\to\R$ we set: $$\mathfrak{u}^{\C}_{\alpha}:=\{Y\in
\mathfrak{u}^{\C}\ |\ \mathrm{for\ all}\ X\in \mathfrak{t},
[X,Y]=2i\pi\alpha(X)\, Y\}.$$ To every root
$\alpha\in R$ we associate the hyperplane
$$\mathcal{H}_{\alpha}=\{X\in\mathfrak{t}\ |\
\alpha(X)=0\}=\ker\alpha\subset \mathfrak{t}.$$
A connected component of $\mathfrak{t}\bs \cup_{\alpha\in
R}\mathcal{H}_{\alpha}$ is called a \emph{Weyl chamber} of the root system
$R$. By definition, it is an open cone of $\mathfrak{t}$. In particular, it
is convex. As we want our notation to be consistent with the statement of the KGS convexity theorem (theorem \ref{KGS_thm}), a Weyl chamber will be denoted by $\mathfrak{t}^*_+\subset
\mathfrak{t}\simeq\mathfrak{t}^*$. Its closure will be denoted by
$\overline{\mathfrak{t}^*_+}$ and called a \emph{closed Weyl chamber}. It
is a closed convex subset of $\mathfrak{t}$.
We now choose a Weyl chamber $\mathfrak{t}^*_+\subset
\mathfrak{t}$ and denote by $R_+(\mathfrak{t}^*_+)$ (or simply $R_+$) the
set of associated positive roots: $$R_+(\mathfrak{t}^*_+):=\{ \alpha\in R\
|\ \alpha(X)> 0 \mathrm{\ for\ one\ and\ therefore\ all}\ X\in
\mathfrak{t}^*_+\}.$$ A positive root $\alpha\in R_+$ is then said to be
\emph{decomposable} if it can be written as a sum $\alpha=\sum_{\beta\in
R_+}n_{\beta}.\beta$ where $n_{\beta}\geq 0$ are integers. Otherwise it is
called indecomposable, or \emph{simple}. We denote by
$\Delta(\mathfrak{t}^*_+)$ (or simply $\Delta$) the set of simple roots in
$R_+$, also called a basis of $R$, since by definition every root
$\alpha\in R$ is of the form $\sum_{\beta\in\Delta}n_{\beta}.\beta$ with
$n_{\beta}\in \Z$. The set of simple roots $\Delta$ associated to the choice of a Weyl chamber
$\mathfrak{t}^*_+$ enables one to give a very nice description of the
polyhedral structure of the closure $\overline{\mathfrak{t}^*_+}$ of the Weyl
chamber, which we will recall shortly. This description is key to the proof
of the momentum convexity theorem presented in \cite{HNP}. We will now describe in
a similar way the polyhedral structure of a Weyl alcove. Generalizing the definition of the hyperplanes $\mathcal{H}_{\alpha}$, we
set, for all $\alpha\in R$ and all $n\in \Z$:
$$\mathcal{H}_{\alpha,n}:=\{X\in\mathfrak{t}\ |\ \alpha(X)=n\}\subset
\mathfrak{t}.$$ The set $$D:=\bigcup_{\alpha\in R,\ n\in
\Z}\mathcal{H}_{\alpha,n}$$ is called the \emph{Stiefel diagram}, or simply diagram, of
$\mathfrak{t}$. It is a family of affine hyperplanes of $\mathfrak{t}$.
\begin{defi}[Weyl alcove]
A connected component $\Walc$ of $\mathfrak{t}\bs D$ is called a \emph{Weyl
alcove} of the root system $R$. By definition, it is an open bounded convex
polyhedron of $\frakt$. For each choice of a Weyl chamber $\mathfrak{t}^*_+\subset
\mathfrak{t}$, with associated set of positive roots $R_+$ and set of simple
roots $\Delta$, there exists a unique Weyl alcove $\Walc$ whose closure
contains $0\in\mathfrak{t}$: $$\Walc =\{ X\in\mathfrak{t}\ |\ \forall
\alpha\in\Delta, \alpha(X)>0\ \mathrm{and}\ \forall \alpha\in R_+\bs\Delta,
\alpha(X)< 1\}.$$ 
\end{defi}
From now on, when we speak of a Weyl alcove, we will assume that its closure contains $0$. We then have:
\begin{prop}\label{alcove}
Let $U$ be a compact connected simply connected Lie group and let
$\Walc\subset \mathfrak{u}=Lie(U)$ be
a Weyl alcove for $U$. Then the set $\exp(\clWalc)\subset U$ is a fundamental domain for the conjugacy
action of $U$ on itself. Moreover, the exponential map $\exp: \mathfrak{u}\to
U$ induces a one-to-one map from the compact convex polytope $\clWalc$ onto
the closed set $\exp(\clWalc)\subset U$. Consequently, we have
homeomorphisms:
$$\overline{\mathcal{W}}\underset{\exp}{\overset{\simeq}{\longto}}
\exp(\overline{\mathcal{W}}) \overset{\simeq}{\longto} U/\Int(U)$$
\end{prop}
 We refer to \cite{Bourbaki} (p.45) or to \cite{Loos} (p.37) for a
proof of this result. We now wish to describe the polyhedral structure of the convex polytope
$\clWalc\subset \mathfrak{t}$. We begin with the polyhedral structure of the
closed Weyl chamber $\mathfrak{t}^*_+$ (see \cite{HNP}). By definition of
$\Delta\subset R_+$, we have: $$\mathfrak{t}^*_+=\{X\in \mathfrak{t}\ |\
\forall \alpha\in \Delta, \alpha(X)>0\}.$$ And: $$\overline{\mathfrak{t}^*_+}=\{X\in \mathfrak{t}\ |\
\forall \alpha\in \Delta, \alpha(X)\geq 0\}.$$ For each subset $S\subset
\Delta$, we set: $$F_S:=\{X\in\mathfrak{t}\ |\ \forall \alpha\in S,
\alpha(X)=0\ \mathrm{and}\ \forall \alpha\in\Delta\bs S,
\alpha(X)>0\}\subset \overline{\mathfrak{t}^*_+}.$$ And we then have:
$$\mathfrak{t}^*_+=F_{\emptyset}\quad \mathrm{and}\quad
\overline{\mathfrak{t}^*_+}=\bigsqcup_{S\subset \Delta} F_S.$$ One remarkable
feature of the sets $(F_S)_{S\subset\Delta}$, which we will call the
\emph{cells} of $\overline{\mathfrak{t}^*_+}$, is that two elements $X,Y$
lying in a same $F_S$ have the same stabilizer $U_X=U_Y$ for the (co-)adjoint
action of $U$ on $\mathfrak{u}\simeq\mathfrak{u}^*$ (see lemma 6.3 in
\cite{HNP}). We will establish an analogous property for the closed
Weyl alcove $\clWalc\subset\overline{\mathfrak{t}^*_+}$ (see
proposition \ref{stabilizer}). We first
observe that $\clWalc$ is also a union of cells. Instead of corresponding to
subsets $S\subset\Delta$ of the set of simple roots, these cells correspond
to subset $S\subset R_+$ of the whole set of positive roots. More precisely,
each subset $S\subset R_+$ can be uniquely written $S=S_1\cup S_2$ where
$S_1\subset\Delta$ and $S_2\subset R_+\bs\Delta$, and for such an $S=S_1\cup
S_2$, we set: $$\Walc_S:=\Bigg\{\ X\in\mathfrak{t}\ |\ \ 
\left\{\begin{array}{rl}
\forall \alpha\in S_1, & \alpha(X)=0 \\
\forall \alpha\in\Delta\bs S_1, & \alpha(X)>0
\end{array}\right.
\quad \mathrm{and} \quad
\left\{\begin{array}{rl}
\forall \alpha \in S_2, & \alpha(X)=1 \\
\forall \alpha \in (R_+\bs\Delta)\bs S_2, & \alpha(X) < 1
\end{array}\right. \quad
\Bigg\}.$$ In particular:$$\Walc = \{X\in\mathfrak{t}\ |\ \forall \alpha\in\Delta, \alpha(X)>0
\ \mathrm{and}\ \forall \alpha\in R_+\bs\Delta, \alpha(X)<1 \}
= \Walc_{\emptyset}.$$ And: $$\clWalc= \{X\in\mathfrak{t}\ |\ \forall
\alpha\in\Delta, \alpha(X)\geq 0
\ \mathrm{and}\ \forall \alpha\in R_+\bs\Delta, \alpha(X) \leq 1 \}
= \bigsqcup_{S\subset R_+} \Walc_S.$$
\begin{defi}\label{cells}
The sets $(\Walc_S)_{S\subset R_+}$ are called the
\emph{cells} of the closed Weyl alcove $\clWalc$.
\end{defi}
We then have the following result:
\begin{lem}\cite{Bourbaki} \emph{(p.48)}\label{centralizer}
Let $U$ be a compact connected simply connected Lie group. Then for any
$u\in U$, the centralizer $U_u=\{v\in U\ |\ vuv^{-1}=u\}$ is connected.
\end{lem}
We can now show that the stabilizer for the conjugacy action of an element $\exp(X)\in\exp(\clWalc)$ depends only on the cell of $\clWalc$ containing $X$:
\begin{prop}\label{stabilizer}
If $u,v\in\exp\clWalc=\sqcup_{S\subset R_+}\exp(\Walc_S)$
lie in a same $\exp(\Walc_S)$, then the centralizers $U_u$ and $U_v$ are
equal.
\end{prop}
\begin{proof}
Since $U$ is compact connected and simply connected, lemma \ref{centralizer}
shows that $U_u$ and $U_v$ are compact connected subgroups of $U$. Therefore
$U_u=U_v$ if and only if their Lie algebras are equal. We then know from
\cite{Loos} (p.7) that the Lie algebra of $U_u$ is:
$$Lie(U_u)\quad = \quad \mathfrak{t} \quad\oplus \sum_{\alpha\, |\,
\exp(i2\pi\alpha(X))=1} \big(\mathfrak{u}\cap\mathfrak{u}_{\alpha}^{\C}\big)$$
where $X\in\mathfrak{t}$ satisfies $\exp(X)=u$. But for $X\in\clWalc$, the
set $$\Big\{\alpha\in R_+\ |\
\exp\big(i2\pi\alpha(X)\big)=1\Big\}$$ is equal to $$\Big\{\alpha\in R_+\ |\
\alpha(X)=0\ \mathrm{or}\ \alpha(X)=1\Big\}$$ so that it only depends on the
cell $\Walc_S\subset\clWalc$ in which $X$ lies, which proves the proposition.
\end{proof}
\begin{defi}\label{US}
For any subset $S\subset R_+$, we denote by $U_S$ the
stabilizer of any element $u\in\exp(\Walc_S)$.
\end{defi}
 Finally, if we consider, for any integer $j$, the set
$\Sigma_j:=\{u\in U\ |\ \dim\, U.u=j\}$ of points of $U$ whose conjugacy
class is of dimension $j$, we have:
\begin{prop}\label{Sigma_j}
The intersection of $\Sigma_j$ with $\exp(\clWalc)$ is:
\begin{eqnarray*}\Sigma_j\cap\exp(\clWalc) \quad & = & \bigsqcup_{S\, |\, \dim\, U_S=\dim\, U-j}
\exp(\Walc_S).\end{eqnarray*}
In addition to that, $\Sigma_j$ is a submanifold of $U$
and so is every $\exp(\Walc_S)$. For any $u\in \exp(\Walc_S)$, one has:
$$T_u\Sigma_j=T_u(U.u)\oplus T_u\exp(\Walc_S)$$
\end{prop}
\begin{proof}
This result is a consequence of the existence of local slices for actions of compact Lie groups (see for instance \cite{GS}): for any local slice $\mathcal{L}$ through $u\in U$, the set of
elements of $\mathcal{L}$ whose conjugacy class has the same dimension as $U.u$ consists precisely of elements which are fixed by $U_u$, and for a linear action the set of such fixed points is a
subspace, so that $\Sigma_j$ is a submanifold of $U$, and it is
$U$-invariant. Now for any $S\subset R_+$,
$\exp(\Walc_S)$ is a submanifold of $U$ (the chart is given by the
exponential map) and it follows from the fact that $\exp(\clWalc)=\sqcup
\exp(\Walc_S)$ is a fundamental domain for the conjugacy action and from
proposition \ref{stabilizer} that \begin{eqnarray*}\Sigma_j\cap\exp(\clWalc)
\quad & = & \bigsqcup_{S\, |\, \dim\, U_S=\dim\, U-j}
\exp(\Walc_S).\end{eqnarray*} 
 and that $\exp(\Walc_S)$ is a local slice through \emph{any}
$u\in\exp(\Walc_S)$, which concludes the proof.
\end{proof}

\subsubsection{\textbf{Inverse image of a certain cell of the Weyl alcove}}\label{inverse_image}

From now on, we will also call \emph{closed Weyl alcove}, or even simply Weyl alcove, the subset $\exp(\clWalc)$ of $U$, which is homeomorphic to the actual closed Weyl alcove $\clWalc$ via the exponential map. It will always be clear from the context which of the sets $\clWalc$ or $\exp(\clWalc)$ we are referring to.
To prove theorem \ref{symp_slice}, we will construct a submanifold $N\subset M$ such that $\mu(N)\subset \exp(\clWalc)\subset U$, so that the map $\widetilde{\mu}:N\to\clWalc$ is no other than
$\widetilde{\mu}=\exp^{-1}\circ\mu|_N : N\to \mathfrak{t}$ and we
will show that it is a smooth map from $N$ to $\mathfrak{t}$. This submanifold $N\subset M$ will in fact be of the form $N=\exp^{-1}(\Walc_S)$ for a certain subset $S\subset R_+$. We will show that $\exp(\Walc_S)$ is the cell of $\exp(\clWalc)$ containing the points of $\mu(M)\cap\exp(\clWalc)$ whose conjugacy class in $U$ is of maximal possible dimension (we will see in lemma \ref{MS} that such a cell is well-defined and unique). This idea is inspired by the usual hamiltonian setting (see \cite{HNP} and also \cite{LMTW}). Studying points of $\mu(M)$ whose conjugacy class is of maximal possible dimension is a natural thing to do as the set of such points is dense in $\mu(M)$ and that we are ultimately interested in studying $\mu(M)$. To begin with, let us define:
$$q:=\max\, \{\dim\, U.\mu(x)\, :\, x\in M\}$$ $$\Sigma_q :=\{u\in U \ |\ \dim\,
U.u =q\}.$$ Recall that $\Sigma_q$ is a submanifold of $U$ (see proposition
\ref{Sigma_j}). Define then: $$M_q:=\{x\in M\ |\ \dim\, U.\mu(x)=q\}
=\mu^{-1}(\Sigma_q)$$ so that $\mu(M_q)$ is exactly the set of points of
$\mu(M)$ whose conjugacy class in $U$ is of maximal possible
dimension.
The first thing to observe is that $M_q$ is an open, connected, and dense
subset of $M$. Let us first show that it is open. For any $x\in M_q$, there
exists an open neighbourhood $\mathcal{V}$ of $\mu(x)$ in $U$ such that for
all $u\in \calV$, $\dim\, U.u \geq \dim\, U.\mu(x)$. Since $\mu$ is
continuous, we have $\mu(y)\in\calV$ for all $y$ in some open set $\calU$ of
$M$ containing $x$. Since $\dim\, U.\mu(x)$ is maximal, we necessarily have
$\dim\, U.\mu(y)=q$ for all $y\in \calU$, so that $\calU \subset M_q$. We now
want to prove that $M_q$ is dense and connected. To that end, we introduce
the set $M_{reg}$ of points of $M$ whose orbit under $U$ is of maximal
possible dimension: $$r:=\max\, \{\dim\, U.x\, :\, x\in M\}$$ $$M_{reg}=\{x\in
M \ |\ \dim\, U.x=r\}.$$ The set $M_{reg}$ is an open, connected, and dense subset of $M$. 
As $M_q$ is open,
the intersection $M_{reg,q}:=M_{reg}\cap M_q$ is a non-empty open set of
$M$. In addition to that, since $(M,\w,\mu:M\to U)$ is a quasi-hamiltonian
space, $M_{reg}$ enjoys the following remarkable property: $$M_{reg}=\{x\in
M\ |\ \mathrm{rk}\, T_x\mu\ \mathrm{is\ maximal}\}.$$ Indeed, since $\mu$ is a momentum map, one has $\im T_x\mu\simeq \fraku_x^{\perp}$ (see \cite{AMM}), and therefore:
\begin{eqnarray*}
\max_{x\in M}\ \{\dim\, \im T_x\mu\} & = & \max_{x\in M}\ \{\dim\,
\mathfrak{u}_x^{\perp}\} \\
& = & \dim\, \mathfrak{u} - \min_{x\in M}\ \{\dim\, \mathfrak{u}_x\} \\
& = & \dim\, U -\min_{x\in M}\ \{\dim\, U_x\} \\
& = & \max_{x\in M}\ \{\dim\, U.x\}\\
& = & r.
\end{eqnarray*}
 In particular, $\mu$ is of constant rank on $M_{reg}$. Now, to
show that $M_q$ is dense and connected in $M$, since we have
$M_{reg,q}\subset M_q\subset M$, it is enough to prove that $M_{reg,q}$ is
dense and connected in $M$. First, we observe that this is locally true in
the following sense:
\begin{lem}\label{localversion}
For all $x\in M_{reg}$, there exists an open neighbourhood $\calV_x$ of
$x\in M_{reg}$ such that $M_{reg,q}\cap\calV_x$ is a dense and connected
subset of $\calV_x$.
\end{lem}
\begin{proof}
Since $\mu$ is of constant rank on $M_{reg}$, there exists an open connected
neighbourhood $\calV_x$ of $x$ in $M_{reg}$ such that $\mu(\calV_x)$ is a
(connected) submanifold of $U$ (of dimension equal to $\mathrm{rk}\, T_x\mu$)
and such that $\mu|_{\calV_x}:\calV_x\to\mu(\calV_x)$ is a locally trivial
submersion onto a connected manifold with connected fibres (the constant
rank theorem says that $\mu$ is locally equivalent to the linear projection
$(x_1,\ {\ldots} \ ,x_n)\mapsto (x_1,\ {\ldots}\ ,x_r,0,\ {\ldots}\ ,0)$). Since $\mu$ is equivariant and
continuous, we have $u.\mu(y)\in\mu(\calV_x)$ for $y\in\calV_x$ and $u$
sufficiently close to $1$ in $U$. Therefore, $\mu(\calV_x)$ is a union of
connected open subsets of conjugacy classes of $U$. Since in addition to that
$\mu(\calV_x)$ is connected, we have, if we set $$q_x:=\max\, \{\dim\, U.v\, :\,
v\in \mu(\calV_x)\}$$ (observe that $q_x$ is not necessarily equal to
$\dim\, U.\mu(x)$) and $$\Omega_x:=\{\mu(y)\in\mu(\calV_x)\ |\ \dim\,
U.\mu(y)=q_x\}$$ that $\Omega_x$ is an open, connected, and dense subset of
$\mu(\calV_x)$. Now, since $\Omega_x$ is an
open dense and connected subset of $\mu(\calV_x)$ and since
$\mu|_{\calV_x}:\calV_x\to\mu(\calV_x)$ is a locally trivial submersion
with connected fibres over the connected manifold $\mu(\calV_x)$, we have
that $(\mu|_{\calV_x})^{-1}(\Omega_x)=\mu^{-1}(\Omega_x)\cap\calV_x$ is an
open dense and connected subset of $\calV_x$ (recall that the submersion
$\mu|_{\calV_x}$ is equivalent to $(x_1,\ {\ldots} \ ,x_n)\mapsto (x_1,\
{\ldots} \ , x_r,0,\ {\ldots} \ ,0)$). Moreover, if $x,y\in M_{reg}$, we can
join them by a path $c:[0,1] \to M_{reg}$. Denote by $\hat{c}$ the compact
connected set $\hat{c}:=c([0,1])$. For every $z\in \hat{c}$, there exists an
open neighbourhood $\calV_z$ of $z$ in $M_{reg}$ such that the set $$R_z:=\big\{
w\in \calV_z\ | \ \dim\, U.\mu(z) =q_z=\max_{v\in\mu(\calV_z)} \{\dim\,
U.v\}\big\}$$ is open, connected, and dense in $\calV_z$. By compactness, we
can cover $\hat{c}$ by a finite number of such $\calV_z$:
$\hat{c}=\calV_{z_1}\cup\ {\ldots} \cup\calV_{z_p}$ with $z_1=x$ and
$z_p=y$. If $\calV_{z_i}\cap\calV_{z_j}\not=\emptyset$, then by density and
openness, $R_{z_i}\cap R_{z_j}\not=\emptyset$. Therefore, for $w\in
R_{z_i}\cap R_{z_j}$, the conjugacy class of $\mu(w)$ has dimension
$q_{z_i}=q_{z_j}$, whence we get $q_x=q_y$, so that $q_x$ is the same for
all $x\in M_{reg}$. As $M_{reg}\cap M_q\not=\emptyset$, one necessarily has
$q_x=q$ for all $x\in M_{reg}$. Therefore $\mu^{-1}(\Omega_x)\cap \calV_x =
M_{reg,q}\cap\calV_x$, which proves the lemma. 
\end{proof}
We now go back to our global study:
\begin{lem}\label{Mq}
The subset $M_{reg,q}:=M_{reg}\cap M_q$ is an open, connected, and dense
subset of $M$. Consequently, so is $M_q$.
\end{lem}
\begin{proof}
$M_{reg,q}$ is open as it is the intersection of two open sets. Since
$M_{reg}$ is dense in $M$, it suffices, in order to prove that $M_{reg,q}$
is dense in $M$, to show that $M_{reg}\subset \overline{M_{reg,q}}$. For any
$x \in M_{reg}$, there exists, by lemma \ref{localversion}, an open
neighbourhood $\calV_x$ of $x$ in $M_{reg}$ such that $M_{reg,q}\cap\calV_x$
is dense in $\calV_x$, so that $x$ is the limit of a sequence of points of
$M_{reg,q}$, which proves that $M_{reg}\subset \overline{M_{reg,q}}$.\\ Let
us now prove that $M_{reg,q}$ is connected. Take $x,y \in M_{reg,q}$. As
$M_{reg}$ is connected, there exists a path $c:[0,1]\to M_{reg}$ joining $x$
to $y$ in $M_{reg}$, and we set $\hat{c}:=c([0,1])$. Then, as in the proof
of lemma \ref{localversion}, there exists a finite open cover
$\hat{c}=\calV_{z_1}\cup\ \ldots \ \cup \calV_{z_p}$ with $z_1=x$ and
$z_p=y$. Then $\calV_{z_1}\cap\calV_{z_i}\not=\emptyset$ for some $i\geq 2$
and by density one has
$M_{reg,q}\cap(\calV_{z_1}\cap\calV_{z_i})\not=\emptyset$. By connectedness
of $\calV_{z_1}\cap M_{reg,q}$, any $z'_i\in
M_{reg,q}\cap(\calV_{z_1}\cap\calV_{z_i})$ can be joined to $z_1$ by a path
in $M_{reg,q}$. Likewise, this $z_i'$ can be joined to $z_i$ by a path in $M_{reg,q}$, therefore $z_1$ can be joined to $z_i$ by a path in $M_{reg.q}$. Repeating this, we obtain a path from $z_1=x$ to $z_p=y$ in
$M_{reg,q}$.\\
Finally, we have proved that $M_{reg,q}$ is connected in $M$, and
we have $M_{reg,q}\subset M_q \subset M=\overline{M_{reg,q}}$, which proves
that $M_q$ is connected in $M$. 
\end{proof}
Instead of describing all of $\mu(M_q)$, what we are really
interested in is describing $\mu(M_q)\cap\exp(\clWalc)\subset\Sigma_q
\cap\exp(\clWalc)$. Recall that $\clWalc$ is a convex polyhedron of
$\mathfrak{t}$, which can de described entirely in terms of roots of $(U,T)$
(see \ref{structure_Walc}). Moreover, we know from proposition
\ref{Sigma_j} that the intersection of $\Sigma_q$ with $\exp(\clWalc)$ is a
finite disjoint union of submanifolds of $U$:
\begin{eqnarray*}
\Sigma_q\cap \exp(\clWalc) & = &
\bigsqcup_{S\ |\ \dim\, U - \dim\, U_S=q} \exp(\Walc_S)\\
& = & \exp(\Walc_{S^{(1)}})
\sqcup \ \ldots\ \sqcup \exp(\Walc_{S^{(m)}})
\end{eqnarray*}
 (where $U_S$ is the stabilizer of any element in $\exp(\Walc_S)$,
see definition \ref{US}), so that we have:
$$\mu(M_q)\cap\exp(\clWalc) \subset \exp(\Walc_{S^{(1)}})
\sqcup \ \ldots\ \sqcup \exp(\Walc_{S^{(m)}})$$ and we now want to study
points in each $\exp(\clWalc_{S^{(i)}})$ which lie in the image of $\mu$. To
that end, we set, for all $i\in \{1,\ {\ldots} \ ,m\}$: $$M_{S^{(i)}}
:=\mu^{-1}\big(\exp(\Walc_{S^{(i)}})\big).$$ By definition we have
$M_{S^{(i)}} \subset M_q$, and since $M_q$ is $U$-stable, we have $u.x\in
M_q$ for all $x\in M_{S^{(i)}}$ and all $u\in U$.
\begin{lem}\label{MSi}
If $M_{S^{(i)}}\not=\emptyset$, it is a submanifold of $M$, and for every
open set $\calO$ of $\MSi$, the set $$U.\calO := \{u.x\, :\,u\in U, x\in
\calO\}$$ is open in $M_q$.
\end{lem}
\begin{proof}
Recall that $\exp(\Walc_{S^{(i)}})$ is a submanifold of $\Sigma_q$ and that
for all $u\in \exp(\Walc_{S^{(i)}})$, one has (see proposition
\ref{Sigma_j}): $$T_u \Sigma_q=T_u\big(U.u\big)\oplus
T_u\big(\exp(\Walc_{S^{(i)}})\big).$$
Moreover, $M_{S^{(i)}}=\mu^{-1}(\exp(\Walc_{S^{(i)}}))$, where $\mu$ is seen
as a map $\mu: M_q\to \Sigma_q$. Hence for all $x\in M_{S^{(i)}}$:
$$T_{\mu(x)}\Sigma_q = T_{\mu(x)}\big(U.\mu(x)\big)\oplus
T_{\mu(x)}\big(\exp(\Walc_{S^{(i)}})\big).$$ But:
$$T_{\mu(x)}\big(U.\mu(x)\big) = T_x\mu.\big(T_x(U.x)\big)\subset \im T_x\mu$$
(where the equality follows from the equivariance of $\mu$), so that, for
all $x\in M_{S^{(i)}}$: $$\im T_x\mu +
T_{\mu(x)}\big(\exp(\Walc_{S^{(i)}})\big)=T_{\mu(x)}\Sigma_q$$ which means
that the map $\mu$ is transverse to the submanifold $\exp(\Walc_{S^{(i)}})$.
By the transversality theorem,
$M_{S^{(i)}}=\mu^{-1}(\exp(\Walc_{S^{(i)}}))$ is therefore a submanifold of
$M_q$, hence of $M$, and one has, for all $x\in M_{S^{(i)}}$:
\begin{eqnarray}\label{imrec}
& T_x M_{S^{(i)}} =
\big(T_x\mu\big)^{-1}\Big(T_{\mu(x)}\big(\exp(\Walc_{S^{(i)}})\big)\Big). & 
\end{eqnarray}
 Moreover, since $\mu(M_q)\subset \Sigma_q$, one has, for such an
$x\in M_{S^{(i)}}\subset M_q$: $$\im T_x\mu \subset T_{\mu(x)}\Sigma_q =
T_{\mu(x)} \big(U.\mu(x)\big)\oplus
T_{\mu(x)}\big(\exp(\Walc_{S^{(i)}})\big).$$ Consequently, for all $v\in T_x
M=T_x M_q$, one has: $$T_x\mu.v=\xi_1+\xi_2$$ where: $$\xi_1 \in
T_{\mu(x)}\big(U.\mu(x)\big)=T_x\mu\big(T_x(U.x)\big) \quad \mathrm{and}\quad \xi_2\in
T_{\mu(x)}\big(\exp(\Walc_{S^{(i)}})\big).$$ Hence $\xi_2=T_x\mu.(v-v_1)$ for
some $v_1\in T_x(U.x)$. Set $v_2:=v-v_1$. Then one has $T_x\mu.v_2 =\xi_2\in
T_{\mu(x)}(\exp(\Walc_{S^{(i)}}))$ and therefore: $$v_2\in
\big(T_x\mu\big)^{-1}\Big(T_{\mu(x)}\big(\exp(\Walc_{S^{(i)}})\big)\Big)=T_x
\MSi.$$ Hence: $$v=v_1 + v_2 \in T_x(U.x) + T_x \MSi.$$ And therefore:
\begin{eqnarray}\label{tangentspace} 
& T_x M = T_x(U.x) + T_x \MSi. &
\end{eqnarray}
One may observe that this sum is generally not a direct
sum, since $\ker T_x\mu\subset T_x\MSi$ and $\ker T_x\mu\cap T_x(U.x)\not=\{0\}$ in general (namely, $\ker T_x\mu\cap T_x(U.x)$ consists of the values at $x\in M$ of fundamental vector fields $X^{\#}_x=\frac{d}{dt}|_{t=0}\exp(tX).x$ for $X$ satisfying $Ad\, \mu(x).X=X$, see \cite{AMM}).
Equality (\ref{tangentspace}) then shows that if $\calO$ is an open subset of $\MSi$ containing $x$,
then $U.\calO$ contains an open subset $\calV_x$ of $M$ containing $x$. Then
for all $u\in U$, $u.\calV_x$ is an open subset of $M$ containing $u.x$ and
contained in $U.\calO$, which shows that $U.\calO$ is open in $M$.
\end{proof}
 In particular, $U.\MSi$ is open in $M$. And we have:
\begin{lem}\label{reunion}
$M_q=\bigsqcup_{i=1}^m U.\MSi$
\end{lem}
\begin{proof}
Let us prove that $M_q \subset \sqcup_{i=1}^m U.\MSi$, the other inclusion
following from the facts that $\MSi\subset M_q$ and that $M_q$ is
$U$-invariant. Consider an element $x\in M_q$. Then $\mu(x)\in \Sigma_q$.
But there exists $u\in U$ such that $u.\mu(x)\in \exp(\clWalc)$, so that
$\mu(u.x) = u.\mu(x) \in \Sigma_q\cap
\exp(\clWalc)=\sqcup_{i=1}^m\exp(\Walc_{S^{(i)}})$ hence $x\in
\sqcup_{i=1}^m U.M_{S^{(i)}}$.
\end{proof}
 Lemmas \ref{MSi} and \ref{reunion} have the following remarkable
consequence:
\begin{lem}\label{MS}
Exactly one of the $\MSi$ is non-empty. 
\end{lem}
\begin{proof}
$M_q=\sqcup_{i=1}^m U.\MSi$ by lemma \ref{reunion}, and $U.\MSi$ is open in
$M_q$ by lemma \ref{MSi}. But $M_q$ is connected by lemma \ref{Mq}, so that
only one of the $U.\MSi$, and consequently only one of the $\MSi$, can be
non-empty.
\end{proof}
\begin{cor}\label{UMS}
It follows from lemmas \ref{reunion} and \ref{MS} that $M_q=U.M_{S^{(i_0)}}$
for a unique $i_0\in \{1,\ {\ldots} \ ,m\}$. It then follows from lemma
\ref{Mq} that $U.M_{S^{(i_0)}}$ is an open, connected, and dense subset of $M$.
\end{cor}
 From now on, we simply denote $S^{(i_0)}$ by $S$. The submanifold
$M_S:=\mu^{-1}(\exp(\Walc_S))$ will end up being our symplectic slice.
We first prove the following result:
\begin{lem}\label{U-action}
If $x\in M_S$ and $u\in U$ are such that $u.x\in M_S$, then $u\in U_S$
(where $U_S$ is the stabilizer of any element in $\exp(\Walc_S)$,
see definition \ref{US}).
\end{lem}
\begin{proof}
If $x\in M_S$ and $u\in U$ are such that $u.x \in M_S$, then $\mu(x)$ and
$u.\mu(x)=\mu(u.x)$ are both elements of $\exp(\Walc_S)\subset
\exp(\clWalc)$, hence $\mu(x)=u.\mu(x)$ that is, $u$ stabilizes some element
of $\exp(\Walc_S)$. Consequently, $u\in U_S$.
\end{proof}
 Together with the fact that $U_S$ is connected, being the
centralizer of an element of a compact connected simply connected Lie group
(see proposition \ref{centralizer}), lemma \ref{U-action} has the following
important consequence:
\begin{lem}\label{MSconnexe}
The manifold $M_S$ is connected.
\end{lem}
\begin{proof}
Assume that $M_S=M_S^{(1)} \sqcup M_S^{(2)}$ is the disjoint union of two
open subsets of $M_S$. Then, by lemma \ref{MSi}, $U.M_S^{(i)}$ is open in
$M$. If $(U.M_S^{(1)})\cap(U.M_S^{(2)})\not=\emptyset$, there exist $x_1\in
M_S^{(1)}$, $x_2\in M_S^{(2)}$ and $u_1,u_2\in U$ such that
$u_1.x_1=u_2.x_2$, hence $u_2^{-1}u_1.x_1=x_2$. But then, by lemma
\ref{U-action}, $u_2^{-1}u_1 \in U_S$, which is connected by proposition
\ref{centralizer}. Therefore, there is a path $(u_t)$ joining $1$ to $u_2^{-1}u_1$
in $U_S$, hence $u_t.x_1$ is a path joining $x_1$ to $x_2$ in $M_S$ (observe that $U_S$ acts on $M_S=\mu^{-1}(\exp(\Walc_S))$, as $\mu$ is an equivariant map), which
contradicts the fact that $x_1$ and $x_2$ lie in disjoint open subsets of $M_S$.
Therefore, $(U.M_S^{(1)})\cap(U.M_S^{(2)})=\emptyset$ and: $$U.M_S =
(U.M_S^{(1)}) \sqcup (U.M_S^{(2)})$$ with $U.M_S^{(i)}$ open in $M$. But
$U.M_S$ is open in $M$  and connected by corollary \ref{UMS}, so that
$U.M_S^{(i)}=\emptyset$ for $i=1$ or $i=2$. Therefore, one of the
$M_S^{(i)}$ is empty, which proves the lemma.
\end{proof}
 We now want to study precisely the relation between $\mu(M_S)$ and
$\mu(M)\cap\exp(\clWalc)$, which was our initial motivation. Recall that
$\mu(M_S)\subset\exp(\Walc_S)\subset \exp(\clWalc)$, the latter being closed
in $U$.
\begin{lem}\label{densite}
If $\mu$ is a closed map (in particular, if $\mu$ is proper), one has:
$\overline{\mu(M_S)}=\mu(M)\cap\exp(\clWalc).$
\end{lem}
\begin{proof}
Take $\mu(x)\in \mu(M)\cap\exp(\clWalc)$. Since $M_q=U.M_S$ is dense in $M$ by
corollary \ref{UMS}, there exist a sequence $(x_j)_{j\in\N}$ of elements of
$M_q$ and a sequence $(u_j)_{j\in\N}$ of elements of $U$ such that $x=\lim\,
x_j$ and $u_j.x_j \in M_S$. Since $U$ is compact, we may assume that $(u_j)$
is convergent and denote its limit by $u:=\lim\, u_j$. Then:
$$u.\mu(x)=\mu(u.x)=\mu\big(\lim\, (u_j.x_j)\big)=\lim\, \mu(u_j.x_j) \in
\overline{\mu(M_S)}.$$ In particular, $u.\mu(x)\in\exp(\clWalc)$, so that
$u.\mu(x)=\mu(x)$, since $\exp(\clWalc)$ is a fundamental domain. Hence
$\mu(x)\in\overline{\mu(M_S)}$, so that $\mu(M)\cap\exp(\clWalc)\subset
\overline{\mu(M_S)}$.\\ Conversely, since $\mu$ is a closed map, $\mu(M)$ is
closed in $U$ and so is $\mu(M)\cap\exp(\clWalc)$. But
$\mu(M_S)\subset\mu(M)\cap\exp(\clWalc)$, hence $\overline{\mu(M_S)}\subset
\mu(M)\cap\exp(\clWalc)$.
\end{proof}
Observe that lemma \ref{densite} is a consequence of corollary
\ref{UMS} and of the fact that $\mu(M_S)\subset \exp(\clWalc)$. This last
point also means that under the identification $U/\Int(U)\simeq\clWalc$,
the map $$\widetilde{\mu}:=p\circ\mu|_{M_S}:M_S\longto U/\Int(U)\simeq
\clWalc\subset\mathfrak{t}=Lie(T)$$ is simply
$\widetilde{\mu}=\exp^{-1}\circ\mu|_{M_S}$. As a matter of fact, it follows
from the definition of $M_S$ that $\mu(M_S)$ lies in the
submanifold $\exp(\Walc_S)$ of $U$, which is diffeomorphic to $\Walc_S$
under $\exp^{-1}$, so that $\widetilde{\mu}=\exp^{-1}\circ\mu|_{M_S}$ is a
smooth map from $M_S$ to $\mathfrak{t}$: 
\begin{center}$\xymatrix{
& {\Walc_S} \ar[d]_{\exp}^{\simeq} \\
{M_S} \ar[r]_{\mu|_{M_S}} \ar[ur]^{\tmu} & {\exp(\Walc_S)}
}$\end{center}
We now compute the differential of $\tmu$,
which is defined to be the composed map $d\widetilde{\mu}:=pr\circ
T\widetilde{\mu}$ of the tangent map $T\tmu : TM_S\to
T\mathfrak{t}\simeq\mathfrak{t}\times\mathfrak{t}$ and the projection
$pr:T\mathfrak{t}\simeq\mathfrak{t}\times\mathfrak{t}\to \mathfrak{t}$ onto
the second factor.
\begin{lem}\label{diff}
The differential $d\tmu$ of $\tmu$ is equal to the $\mathfrak{t}$-valued
$1$-form $\mu^*\theta$ on $M_S$, where $\theta$ is the Maurer-Cartan
$1$-form on $T$, that is, the $\mathfrak{t}$-valued $1$-form defined for
$t\in T$ and $\xi\in T_tT$ by $\theta_t(\xi)=t^{-1}.\xi=\xi.t^{-1}$:
$$d\tmu = \mu^*\theta.$$
\end{lem}
\begin{proof}
Recall that the tangent map to the exponential map $\exp:\mathfrak{u}\to U$
is given, for all $X\in\mathfrak{u}$ and all $\xi\in
T_X\mathfrak{u}=X+\mathfrak{u}$, by: $$T_X\exp.\xi =
\exp(X).\Big(\frac{1-e^{-ad\, X}}{ad\, X}\cdot(\xi-X)\Big)$$
where $\frac{1-e^{-ad\, X}}{ad\, X}$ is the endomorphism
of $\mathfrak{u}$ given, for all $\zeta\in\mathfrak{u}$, by:
$$\frac{1-e^{-ad\, X}}{ad\, X}\cdot\zeta = \sum_{k=1}^{+\infty}\frac{(-ad\,
X)^{k-1}}{k!}\cdot\zeta$$ and where $\exp(X).\zeta$ denotes the effect on
tangent vectors $\zeta\in\mathfrak{u}=T_1U$ of the left translation of
element $\exp(X)$ in $U$. In the present case, we have to consider
$\exp:\mathfrak{t}\to T$ with $T$ abelian, since for
$x\in M_S$, we have $\mu(x)\in\exp(\Walc_S)\subset T$,
so that: $$\frac{1-e^{-ad\,
X}}{ad\, X}.\zeta = \zeta$$ as $(ad\, X)^{k-1}.\zeta=0$ as soon as $k-1\geq
1$. Then, for all $X\in \mathfrak{t}$ and all $\xi\in
T_X\mathfrak{t}=X+\mathfrak{t}$: $$T_X\exp.\xi=\exp(X).(\xi-X).$$ Therefore, for all
$x\in M_S$ and all $v\in T_xM_S$, we have:
\begin{eqnarray*}
T_x(\underset{=\mu}{\underbrace{\exp\circ\tmu}}).v & = & T_{\tmu(x)}\exp\circ
T_x\tmu . v \\
& = &
\exp\big(\tmu(x)\big).\big(\underset{=(d\tmu)_x.v}{\underbrace{T_x\tmu.v
-\tmu(x)}}\big).
\end{eqnarray*}
 So that: $$T_x\mu.v =
\exp\big(\tmu(x)\big).\big((d\tmu)_x.v\big).$$ Hence:
\begin{eqnarray*}
(d\tmu)_x.v & = & \Big(\exp\big(\tmu(x)\big)\Big)^{-1}.(T_x\mu.v)\\
& = & \theta_{\exp\circ\tmu(x)}(T_x\mu.v)\\
& = & \theta_{\mu(x)}(T_x\mu.v)\\
& = & (\mu^*\theta)_x.v.
\end{eqnarray*}
\end{proof}
We may now prove theorem \ref{symp_slice}:
\begin{proof}[\textbf{Proof of theorem \ref{symp_slice} (Existence of a
symplectic slice)}]
Let us set $N:=M_S$, where $M_S= $ $\mu^{-1}(\exp(\Walc_S))$ is the submanifold of $M$
constructed above.\\
(i) Lemma \ref{MSconnexe} shows that $N$ is connected.\\
(ii) Since $N=M_S=\mu^{-1}(\exp(\Walc_S))$ with $\mu$ equivariant,
and since the conjugacy action of $T$ on $\exp(\Walc_S)$ is trivial (as $T$
is abelian), we have that $N$ is $T$-stable.\\
(iii) Let us show that $\w|_{M_S}$ is a symplectic form. We denote by
$i$ the inclusion map $i:M_S\hookrightarrow M$, so that $i^*\w=\w|_{M_S}$.
Recall that $\qhamsp$ is a quasi-hamiltonian space, therefore $d\w=-\mu^*\chi$, where $\chi$ is the Cartan $3$-form of $U$ ($\chi_1(X,Y,Z)=(X\, |\, [Y,Z])$ for all $X,Y,Z\in \fraku=T_1U$).
First, we have: $$d(i^*\w)=i^*(d\w)=i^*(-\mu^*\chi)=-(\mu\circ i)^*\chi.$$
But $\mu\circ i=\mu|_{M_S}$ is $T$-valued and $\chi|_{T}=0$ as $T$ is abelian.
Therefore, $d(i^*\w)=0$. Second, let us show that $i^*\w$ is non-degenerate.
Take $x\in M_S$ and $v\in T_xM_S$ such that for all $w\in T_xM_S$,
$\w_x(v,w)=0$. In particular, $v\in (T_xM_S)^{\perp_{\w}}\subset T_xM$. But
we know from lemma \ref{MSi} (see (\ref{imrec}))  that:
$$T_xM_S=\big(T_x\mu\big)^{-1}\big(T_{\mu(x)}\exp(\Walc_S)\big).$$ Hence: $$\ker
T_x\mu=T_x\mu^{-1}(\{0\}) \subset T_xM_S.$$ And therefore:
$$(T_xM_S)^{\perp_{\w}} \subset (\ker T_x\mu)^{\perp_{\w}}.$$
Since $\qhamsp$ is a quasi-hamiltonian space, we have (see \cite{AMM}): 
\begin{eqnarray}\label{kernel_and_orbit}
(\ker T_x\mu)^{\perp_{\w}} =
\{X^{\#}_x\ : \ X\in \mathfrak{u}\} = T_x (U.x).
\end{eqnarray}
Take now $X\in\mathfrak{u}$ such that $v=X^{\#}_x$. Then, by equivariance of $\mu$:
$$T_x\mu.v=T_x\mu.X^{\#}_x=X^{\dagger}_{\mu(x)} \in
T_{\mu(x)}\big(\exp(\Walc_S)\big)\cap T_{\mu(x)}\big(U.\mu(x)\big)=\{0\}$$
(where $X^{\dagger}$ denotes the fundamental vector field associated to
$X\in\mathfrak{u}$ by the conjugacy action of $U$ on itself, and where the
last equality follows from proposition \ref{Sigma_j}). Hence $v\in \ker
T_x\mu$. But by equality (\ref{kernel_and_orbit}), one has $\ker T_x\mu\subset
(T_x(U.x))^{\perp_{\w}}$, therefore: $$v\in
\big(T_x(U.x)\big)^{\perp_{\w}}\cap \big(T_x M_S)^{\perp_{\w}}.$$ And we know
from proposition \ref{MSi} (see (\ref{tangentspace})) that: $$T_xM_S + T_x(U.x) = T_x M.$$ Therefore $v\in (T_xM)^{\perp_{\w}}=\ker \w_x$. Then
$v\in \ker T_x\mu \cap \ker\w_x$, which is always equal to $\{0\}$ (see \cite{AMM}).\\
(iv) Let us now show that the action of $T$ on $M_S$ is hamiltonian
with momentum map $\tmu=\exp^{-1}\circ\mu:M_S\to\mathfrak{t}$. Denore by $\theta^L=u^{-1}.du$ and $\theta^R=du.u^{-1}$ the left and right Maurer-Cartan $1$-forms of $U$. Take
$X\in\mathfrak{t}$. Since $\qhamsp$ is a quasi-hamiltonian space, we have:
$$\iota_{X^{\#}}\w =\frac{1}{2}\mu^*(\theta^L+\theta^R\, |\,X).$$ Therefore,
for all $x\in M_S$ and all $v\in
T_xM_S=(T_x\mu)^{-1}(T_{\mu(x)}\exp(\Walc_S))$:
$$(\iota_{X^{\#}}\w)_x.v = \frac{1}{2}\big(\theta^L_{\mu(x)}(T_x\mu.v) +
\theta^R_{\mu(x)}(T_x\mu.v)\, |\,X\big).$$ But $T_x\mu.v\in
T_{\mu(x)}\exp(\Walc_S)$ with $\exp(\Walc_S)\subset T$, so that, since $T$
is abelian: $$\theta^L_{\mu(x)}(T_x\mu.v) = \theta^R_{\mu(x)}(T_x\mu.v) =
\theta_{\mu(x)}(T_x\mu.v) = (\mu^*\theta)_x.v$$ where $\theta$ is the
Maurer-Cartan $1$-form of $T$. Hence: $$(\iota_{X^{\#}}\w)_x.v = \big(
(\mu^*\theta)_x.v\, |\,X\big) = \big( (d\tmu)_x.v\, |\,X\big)$$ where the last
equality follows from lemma \ref{diff}. Denote by $(\tmu\, |\,X)$ the
function:
$$\begin{array}{rrcl}
(\tmu\, |\,X): & M_S & \longto & \R \\
& x & \longmapsto & (\tmu(x)\, |\,X)
\end{array}$$
 (where $\tmu=\exp^{-1}\circ\mu|_{M_S}:M_S\to \mathfrak{t}$). We
then have: $$\big(d(\tmu\, |\,X)\big)_x.v = \big( (d\tmu)_x.v\, |\,X\big).$$
Therefore, for all $X\in\mathfrak{t}$: $$\iota_{X^{\#}}\w = d(\tmu\, |\,X).$$
That is: the hamiltonian vector field associated to the function $(\tmu\ |\
X)$ is the fundamental vector field $X^{\#}$, which shows that the action of
$T$ on $M_S$ is hamiltonian.\\
(v) Corollary \ref{UMS} shows that $\overline{U.M_S}=M$. Since $\mu$
is a proper map, lemma \ref{densite} shows that
$\overline{\mu(M_S)}=\mu(M)\cap\exp(\clWalc)$, or equivalently:
$\overline{\tmu(M_S)}=\tmu(M)$.
\end{proof}

\subsection{Momentum convexity for quasi-hamiltonian actions}\label{proof_AMMW}

In the previous subsection, we gave a detailed explicit construction of a symplectic slice in an arbitrary quasi-hamiltonian $U$-space $\qhamsp$, with $U$ simply connected. We did so because, in the next subsection, having a symplectic slice of the form $M_S=\exp^{-1}(\Walc_S)$ where $\Walc_S$ is a cell of the closed Weyl alcove $\clWalc$ will be crucial to our proof of the real convexity theorem \ref{real_convexity_thm}. Incidentally, it enables us to give a proof of the convexity theorem \ref{AMMW_thm}. This proof is a quasi-hamiltonian analogue of the proof of the Kirwan theorem given in \cite{HNP}.  It rests on the following lemma:
\begin{lem}\label{conveximage}
Let $M_S:=\mu^{-1}(\exp(\Walc_S))$ be a symplectic slice
for the connected quasi-hamiltonian space $(M,\w,\mu:M\to U)$. Then the set
$\mu(M_S)\subset\exp(\clWalc)\simeq\clWalc$ is a convex polytope.
\end{lem}
\begin{proof}
By theorem \ref{local_convexity_thm}, the map $\mu|_{M_S}$ gives rise to
local convexity data $(\calV_x,C_{\mu(x)})_{x\in M_S}$. But we cannot conclude immediately that $\mu(M_S)$ is convex because $\mu|_{M_S}$ is not a proper map in general (as $M_S=\mu^{-1}(\exp(\Walc_S))$ is in general not closed in $M$). Instead, we proceed as in \cite{HNP}: write the convex set
$\Walc_S$ as an increasing sequence of \emph{closed} locally polyhedral
convex subsets $(D_n)_{n\in \N}$. Then: $$\exp(\Walc_S)=\bigcup_{n\in\N}
\exp(D_n).$$  Therefore, proposition \ref{ascending_sequence} applies to the closed sets
$Y_n:=\mu^{-1}(\exp(D_n))$ and $\mu|_{Y_n}$ gives rise to local convexity
data $(\calV_x,C_{\mu(x)}\cap\overline{\R^+.(\exp(D_n)\bs\{\mu(x)\})})_{x\in
Y_n}$. Additionally, since $Y_n$ is closed in $M_S$, $\mu|_{Y_n}$ is a
proper map. Since $M_S$ is connected and is an increasing union of closed
subsets
$M_S=\cup_{n\in \N} Y_n$, we can find an ascending sequence $(Z_n)_{n\in\N}$
of connected components of the $(Y_n)_{n\in\N}$ such that $M_S=\cup_{n\in\N}
Z_n$. Each $Z_n$ is closed in $Y_n$, so that $\mu|_{Z_n}$ is a proper map
which gives rise to local convexity data $(\calV_x,C_{\mu(x)}\cap\overline{\R^+.(\exp(D_n)\bs\{\mu(x)\})})_{x\in Z_n}$. Therefore, by theorem
\ref{LKP_thm}, $\mu(Z_n)$ is a convex polytope. We then have that $\mu(M_S)$ is an
increasing union $\mu(M_S)=\cup_{n\in\N} \mu(Z_n)$ of convex subpolytopes of
$\exp(\clWalc)\simeq\clWalc$, which implies that it is a convex polytope.
\end{proof}
\noindent We can now prove theorem \ref{AMMW_thm}:
\begin{proof}[\textbf{Proof of the convexity theorem \ref{AMMW_thm}}]
We have $\mu(M)\cap\exp(\clWalc)=\overline{\mu(M_S)}$ by theorem
\ref{symp_slice} and $\mu(M_S)$ is a convex polyhedron by lemma \ref{conveximage}, hence so is
$\overline{\mu(M_S)}$. Let us now prove that the map $\mu:M\to U$ has connected fibres. First write, as in the proof of lemma \ref{conveximage}, $M_S=\cup_{n\in \N}Z_n$. Since $\mu|_{Z_n}$ is a proper map giving rise to local convexity data $(\calV_x,C_{\mu(x)})_{x\in Z_n}$, the local-global principle \ref{LKP_thm} shows that the fibres of $\mu|_{Z_n}$ are connected. Since $M_S=\cup_{n\in \N}Z_n$, the fibres of $\mu|_{M_S}:M_S\to \exp(\clWalc)$ are connected. And since $\mu(M_S)$ is convex in $\exp(\clWalc)\simeq\clWalc$, lemma \ref{connexite_des_fibres} shows that $\mu|_{\overline{M_S}}:\overline{M_S}\to\exp(\clWalc)$ has connected fibres. Second, the fact that $\mu$ is $U$-equivariant shows that the fibres of $\mu$ above $\mu(U.\overline{M_S})$ are connected. But since $U$ is compact $U.\overline{M_S}=\overline{U.M_S}$. But then, by theorem \ref{symp_slice} one has $\overline{U.M_S}=M$, so that $\mu:M\to U$ has connected fibres.
\end{proof}

\subsection{Image of the fixed-point set of a form-reversing involution}\label{main_proof}

We can now give a proof of theorem \ref{real_convexity_thm}. The idea is to apply Duistermaat's theorem \ref{Duist_thm} to the symplectic slice $M_S=\exp^{-1}(\Walc_S)\subset M$ constructed in subsection \ref{slice_construction}. This will not work directly for, as we already noted in subsection \ref{proof_AMMW}, the momentum map $\mu|_{M_S}$ is not proper in general. But this idea will guide in the following. To apply Duistermaat's theorem, we need to verify the following:
\begin{enumerate}
\item the involution $\beta$ leaves the symplectic slice $M_S$ globally stable.
\item the involution $\beta_S:=\beta|_{M_S}$ is anti-symplectic and compatible with the torus action on $M_S$ and with the momentum map of this action.
\item the fixed-point set of $\beta|_{M_S}$ is non-empty.
\end{enumerate}
The proof of these three facts relies very strongly, as we shall see, on the following two assumptions:
\begin{enumerate}
\item[(H1)] the symmetric pair $(U,\tau)$ is of maximal rank. In particular, there exists a closed Weyl alcove satisfying $\exp(\clWalc)\subset Fix(\taum)$.
\item[(H2)] the symplectic slice $M_S\subset M$ is the pre-image of a face of such a closed Weyl alcove. 
\end{enumerate}
\newcommand{\slice}{\mu^{-1}(\exp(\Walc_S))}
Observe that the construction of the symplectic slice $M_S=\slice$ depends on the choice of the Weyl alcove $\Walc\subset\frakt$, hence on the choice of a maximal torus $T$. From now on, we will always choose such a maximal torus $T\subset U$ to satisfy $T\subset Fix(\taum)$. In particular, $\tau(t)=t^{-1}$ for all $t\in T$, so that for all $x\in M_S$ and all $t\in T$ one has $\beta(t.x)=\tau(t).\beta(x)=t^{-1}.\beta(x)$. Moreover, one has $\exp(\Walc_S)\subset \exp(\clWalc)\subset T\subset Fix(\taum)$ hence for all $x\in M_S=\slice$, $\mu\circ\beta(x)=\taum\circ\mu(x)=\mu(x)$, which means that $\beta|_{M_S}$ is compatible with the action of $T$ on $M_S$ and with the momentum map of this action in the sense of theorem \ref{Duist_thm}. The fact that $\beta|_{M_S}$ is anti-symplectic is immediate since $\beta^*\w=-\w$ on $M$ and the symplectic form on $M_S$ is $\w|_{M_S}$. Thus, the only thing left to verify is that $Fix(\beta|_{M_S})\not=\emptyset$.\\
Recall that we assumed, in the statement of theorem \ref{real_convexity_thm}, that $\mu(M^{\beta})\cap Q_0\not=\emptyset$, where $Q_0$ is the connected component of $1$ in $Fix(\taum)$, and that we denoted by $L$ a connected component of $M^{\beta}$ such that $\mu(L)\cap Q_0\not=\emptyset$ (hence such that $\mu(L)\subset Q_0$). Observe that the neutral component $K$ of $U^{\tau}=Fix(\tau)\subset U$ acts on $L$ since $\beta(k.x)=\tau(k).\beta(x)=k.x$ for all $k\in K\subset Fix(\tau)$ and all $x\in L\subset Fix(\beta)$. As a matter of fact, since here the compact group $U$ is assumed to be simply connected, the group $U^{\tau}$ is necessarily connected (see for instance \cite{Loos}), but this is not very important here. Recall from subsection \ref{slice_construction} that $q=\max_{x\in M}\{\dim\, U.\mu(x)\}$ and that $M_q=\{x\in M\ |\ \dim\, U.\mu(x) = q\}$. One then has:
\begin{lem}\label{pts_image_reg_ds_L}
The set $L_q:=L\cap M_q=\{x\in L\ |\ \dim\, U.\mu(x)=q\}$ is dense in $L$ (in particular, it is non-empty).
\end{lem}
\begin{proof}
Let us set $q':=\max\, \{\dim\, K.\mu(x)\, :\, x\in L\}$ and $L^{(K)}_{q'}:=\{x\in L\ |\ \dim\, K.\mu(x)=q'\}$. Then by definition, $L^{(K)}_{q'}$ is non-empty. Let us prove that it is an open and dense subset of $L$. To that end, denote by $\Lambda$ the set of points of $U.L^{(K)}_{q'}$ whose $U$-orbit is of maximal possible dimension. Then $\Lambda$ is an open and dense subset of $U.L^{(K)}_{q'}$. Take now $x\in L^{(K)}_{q'}$ and consider an open neighbourhood $\calV$ of $x$ in $L^{(K)}_{q'}$. Then $U.\calV$ is open in $U.L^{(K)}_{q'}$ and therefore intersects $\Lambda$. If $y\in \Lambda\cap U.\calV$, its $U$-orbit intersects $L$ in a point of $\Lambda\cap\calV$. Consequently, $\Lambda\cap L^{(K)}_{q'}$ is dense in $L^{(K)}_{q'}$. Further, the momentum map $\mu$ is of constant rank on the set $\Lambda\cap L^{(K)}_{q'}$, since all $U$-orbits of points in this set have the same dimension and since $\mathrm{rk}\, T_x\mu=\dim\, \fraku_x^{\perp}=\dim\, U.x$ (see \cite{AMM}). We can then proceed exactly as in lemmas \ref{localversion} and \ref{Mq} and obtain that $\Lambda\cap L^{(K)}_{q'}$, and therefore $L^{(K)}_{q'}$, is dense in $L$.\\ We will conclude the proof by showing that $L^{(K)}_{q'}=L\cap M_q$. Take $x\in L$. Then $\mu(x)\in Q_0$. Since the symmetric pair $(U,\tau)$ is of maximal rank, $\exp(\clWalc)$ is a fundamental domain for the action of $K$ on $Q_0$. This implies that if $w\in Q_0$, then $\dim\, K.w$ is maximal if and only if $\dim\, U.w$ is maximal. Indeed, $\dim\, K.w=\dim\, K -\dim\, (U_{S'})^{\tau}$, where $S'$ is the uniquely defined set such that $K.w\cap\exp(\Walc_{S'})\not=\emptyset$ and $U_{S'}$ is the stabilizer of any point in $\exp(\Walc_{S'})$, and $\dim\, U.w=\dim\, U-\dim\, U_{S'}$, \emph{for the same} $S'$ since $U.w\supset K.w$ and the fundamental domains of the $U$ action on $U$ and the $K$-action on $Q_0$ are the same. Therefore, here, $\dim\, K.\mu(x)=q'$ if and only if $\dim\, U.\mu(x)=q$, which means that $L^{(K)}_{q'}=L\cap M_q$. This concludes the proof of the proposition, as we have seen that $L^{(K)}_{q'}$ is dense in $L$. 
\end{proof}
Lemma \ref{pts_image_reg_ds_L} means that the subset of points of $L$ whose image is regular in $\mu(M)$ is (non-empty and) dense in $L$. A similar result was established earlier in lemma \ref{Mq} for points of $M$ whose image is regular. Observe that the fact that $(U,\tau)$ is of maximal rank is crucial: in the terminology of appendix B of \cite{OSS}, the root system of $(U,T)$ and the system of \emph{restricted roots} of $(U,\tau)$ are the same (recall that we assumed $T\subset Fix(\taum)$) and therefore $\dim\, K.w$ is maximal if and only if $\dim\, U.w$ is maximal, for any $w\in Q_0$. Let us now carry on with the proof of theorem \ref{real_convexity_thm}.
\begin{lem}\label{pts_fixes_betaS}
One has: $M_S^{\beta}\not=\emptyset$ and $L_q=K.M_S^{\beta}$.
\end{lem}
\begin{proof}
Lemma \ref{pts_image_reg_ds_L} shows that $L_q\not=\emptyset$. Take now $x\in L_q$. Then $\mu(x)\in Q_0$. Since the symmetric pair $(U,\tau)$ is of maximal rank, $\exp(\clWalc)$ is a fundamental domain for the action of $K$ on $Q_0$, hence the existence of a $k$ in $K$ such that $k\mu(x)k^{-1}\in \exp(\clWalc)\cap \mu(M_q)=\exp(\Walc_S)$, where the last equality follows from lemma \ref{MS}. Hence $(k.x)\in \slice=M_S$. Moreover, $\beta(k.x)=\tau(k).\beta(x)=k.x$, hence $(k.x)\in M_S^{\beta_S}$, which is therefore non-empty, and we have indeed: $L_q=K.M_S^{\beta_S}$.
\end{proof}
Recall that $\mu(M_S^{\beta_S})\subset \exp(\clWalc)\subset T\subset Q_0$ by construction of the symplectic slice $M_S$. Since $K$ is connected and $M_S^{\beta_S}$ is not connected in general, we observe that $L_q$ is not connected in general (this could be seen from the proof of lemma \ref{pts_image_reg_ds_L}: the set $\Lambda$ constructed there had no reason to be connected in general). We can now prove an analogue of lemma \ref{densite} (or equivalently, of statement (v) of proposition \ref{symp_slice}):
\begin{lem}\label{density_of_mu_FixBeta}
If $\mu$ is a closed map (in particular, if $\mu$ is proper), one has: $\overline{\mu(M_S^{\beta_S})}=\mu(L)\cap\exp(\clWalc)$.
\end{lem}
\begin{proof}
Take $\mu(x)\in\mu(L)\cap\exp(\clWalc)$. Then, by lemma \ref{pts_image_reg_ds_L}, $x=\lim\, x_j$ with $x_j\in L_q$, and by lemma \ref{pts_fixes_betaS}, for all $j$ there exists an element $k_j\in K$ such that $(k_j.x_j)\in M_S^{\beta_S}$. Since $K=Fix(\tau)\subset U$ is compact, we may assume that the sequence $(k_j)$ converges to a certain $k\in K$. Then: $k.\mu(x)=\mu(k.x)=\lim\, \mu(k_j.x_j)\in\overline{\mu(M_S^{\beta_S})}$. In particular, $k.\mu(x)\in\exp(\clWalc)$, so that $k.\mu(x)=\mu(x)$, since $\exp(\clWalc)$ is a fundamental domain for the action of $K$ on $Q_0$. Therefore, $\mu(x)\in\overline{\mu(M_S^{\beta_S})}$, so that $(\mu(L)\cap \exp(\clWalc))\subset\overline{\mu(M_S^{\beta_S})}$.\\ Conversely, since $\mu$ is a closed map and $L$ is closed in $M$ (recall that $L$ is a connected component of $M^{\beta}$), $\mu(L)$ is closed in $U$. But $\mu(M_S^{\beta_S})\subset(\mu(L)\cap\exp(\clWalc))$ by construction of $M_S$, so that $\overline{\mu(M_S^{\beta_S})}\subset(\mu(L)\cap \exp(\clWalc))$.
\end{proof}
Thus, $\mu(M_S^{\beta_S})$ is almost the whole of $\mu(L)\cap\exp(\clWalc)$. This is interesting because we can now relate $\mu(M_S^{\beta_S})$ to $\mu(M_S)$ (the latter being dense in $\mu(M)\cap\exp(\clWalc)$ by lemma \ref{densite}) in the following way, which is essentially Duistermaat's statement applied to the symplectic slice $M_S$:
\begin{prop}\label{Duist-like}
Assume that $\mu:M\to U$ is a proper map. Then, in the above notations: $$\mu(M_S^{\beta_S})=\mu(M_S).$$
\end{prop}
As in the proof of lemma \ref{conveximage}, we cannot apply theorem \ref{Duist_thm} directly to the symplectic slice $M_S$, for $\mu|_{M_S}$ is in general not proper. But we may work with the ascending sequence $(Z_n)_{n\in\N}$ introduced in the proof of lemma \ref{conveximage}: $\Walc_S$ is an ascending union of closed convex subsets $\Walc_S=\cup_{n\in\N}D_n$ and $M_S=\slice$ is an ascending union of \emph{closed connected} sets $Z_n\subset \mu^{-1}(D_n)$. By proposition \ref{ascending_sequence}, the map $\tmu|_{Z_n}$ is a proper map which gives rise to local convexity data $(\calV_x,C_{\tmu(x)})_{x\in Z_n}$ and, by theorem \ref{LKP_thm}, the set $\tmu(Z_n)$ is convex. We then observe the following fact:
\begin{lem}\label{towards_Duist-like}
Consider $n\in\N$ such that $Z_n\cap M_S^{\beta_S}\not=\emptyset$. Then for any connected component $W\subset(Z_n\cap M_S^{\beta_S})$, the set $\tmu(W)$ is convex.
\end{lem}
\begin{proof}
First, observe that such an $n\in N$ always exists since $M_S=\cup_{n\in\N}Z_n$ and $M_S^{\beta_S}\not=\emptyset$ by lemma \ref{pts_fixes_betaS}. Second, observe that $W$ is closed in $Z_n$. As $\tmu|_{Z_n}$ is a closed map, $\tmu(W)$ is a closed connected subset of $\frakt$. Take then $x\in W$. It follows from point (iii) of the local convexity theorem \ref{local_convexity_thm} that there exists an open neighbourhood $\calO_{\tmu(x)}$ of $\tmu(x)\in\frakt$ such that $\tmu(\calV_x\cap W)=C_{\tmu(x)}\cap \calO_{\tmu(x)}$, where $(\calV_x,C_{\tmu(x)})_{x\in Z_n}$ is the local convexity data of the map $\tmu|_{Z_n}$. Further, $\tmu(W)$ is contained in the convex set $\tmu(Z_n)$, therefore in 
\begin{eqnarray}\label{step1}
\tmu(x)+\overline{\R^+.(\tmu(Z_n)\bs\{\tmu(x)\})}=C_{\tmu(x)}
\end{eqnarray}
where the last equality follows from theorem \ref{LKP_thm}. Hence: $$C_{\tmu(x)}\cap \calO_{\tmu(x)}=\tmu(W)\cap \calO_{\tmu(x)}$$
so that $\tmu(W)$ is convex by corollary \ref{real_Prinzip}. In fact, corollary \ref{real_Prinzip} also shows that for all $x\in W$:
\begin{eqnarray}\label{step2}
C_{\tmu(x)}=\tmu(x)+\overline{\R^+.(\tmu(W)\bs\{\tmu(x)\})}
\end{eqnarray}
\end{proof}
We may now prove proposition \ref{Duist-like}:
\begin{proof}[\textbf{Proof of proposition \ref{Duist-like}}]
Consider $n\in\N$ such that $Z_n\cap M_S^{\beta_S}\not=\emptyset$ and let $W$ be a connected component of $Z_n\cap M_S^{\beta_S}$. Then we know from lemma \ref{towards_Duist-like} that $\tmu(W)\subset\tmu(Z_n)$ is an inclusion between two convex sets of a finite-dimensional vector space. Additionally, by comparing (\ref{step1}) and (\ref{step2}), we obtain:
$$\tmu(x)+\overline{\R^+.(\tmu(W)\bs\{\tmu(x)\})} = \tmu(x)+\overline{\R^+.(\tmu(Z_n)\bs\{\tmu(x)\})}.$$
Therefore, lemma \ref{inclusion_de_deux_convexes} applies and we get $\tmu(W)=\tmu(Z_n)$. Consequently, $\tmu(Z_n\cap M_S^{\beta_S})=\tmu(Z_n)$. Since $M_S=\cup_{n\in\N} Z_n$, one has $M_S^{\beta_S}=\cup_{n\in \N}(Z_n\cap M_S^{\beta_S})$ and therefore: $$\tmu(M_S^{\beta_S}) = \bigcup_{n\in\N} \tmu(Z_n\cap M_S^{\beta_S}) = \bigcup_{n\in\N}\tmu(Z_n) = \tmu(M_S).$$ Since $\tmu(M_S)$ is contained in $\exp(\clWalc)$ and $\tmu=\exp^{-1}\circ\mu|_{M_S}$, the above equality is equivalent to $\mu(M_S)=\mu(M_S^{\beta_S})$.
\end{proof}
We can now prove the real convexity theorem \ref{real_convexity_thm}:
\begin{proof}[\textbf{Proof of the real convexity theorem \ref{real_convexity_thm}}]
Since $\tmu$ is a proper map, one may apply successively lemma \ref{density_of_mu_FixBeta}, proposition \ref{Duist-like} and  lemma \ref{densite} and obtain: $$\mu(L)\cap\exp(\clWalc) = \overline{\mu(M_S^{\beta_S})} = \overline{\mu(M_S)} = \mu(M) \cap \exp(\clWalc)$$ where $L$ is any connected component of $M^{\beta}$ satisfying $\mu(L)\cap Q_0\not=\emptyset$.
\end{proof}
As an immediate corollary of the above real convexity theorem, we point out the following result, which we will use later in applications (see section \ref{applications}).
\begin{cor}\label{fibres_and_fix_beta}
For all $t\in \exp(\clWalc)$, $\mu^{-1}(\{t\})\not=\emptyset$ if and only if $\mu^{-1}(\{t\})\cap M^{\beta}\not=\emptyset$. In particular, $1\in \mu(M^{\beta})$ if and only if $1\in\mu(M)$.
\end{cor}
\begin{proof}
Assume that $\mu^{-1}(\{t\})\not=\emptyset$. Then $t\in\mu(M)\cap\exp(\clWalc)=\mu(M^{\beta})\cap\exp(\clWalc)$ by theorem \ref{real_convexity_thm}, so that there exists $y\in M^{\beta}$ satisfying $\mu(y)=t$. The converse implication is obvious.
\end{proof}
Observe that, by equivariance of $\mu$, the above statement is still true for $t\in K.\exp(\clWalc)=Q_0$. It can obviously not be true for any $t\in U$ since one necessarily has $\mu(M^{\beta})\subset Fix(\taum)$.

\section{Applications}\label{applications}

In this final section, we give some applications of the real convexity theorem \ref{real_convexity_thm}. The first application is the construction of lagrangian submanifolds in quasi-hamiltonian quotients. The quasi-hamiltonian quotient associated to a quasi-hamiltonian space $\qhamsp$ is the set $M//U:=\fibre/U$. Alekseev, Malkin and Meinrenken showed in \cite{AMM} that if the action of $U$ on $\fibre$ is free then $M//U$ is a symplectic manifold (and that if the action is locally free then $M//U$ is a symplectic orbifold). To make the presentation simpler, we will always assume in the following  that $M//U$ is a smooth manifold. To obtain a lagrangian submanifold of $M//U$, one may for instance construct an anti-symplectic involution $\sigma$ on $M//U$. Then, if $Fix(\sigma)$ is non-empty it is a lagrangian submanifold of $M//U$. To obtain such a $\sigma$, one may start from an involution $\beta$ on the quasi-hamiltonian space $M$. Compatibility with the action of $U$ and the momentum map ensures that $\beta$ induces an anti-symplectic involution $\bhat([x]):=[\beta(x)]$ on the orbit space $M//U=\fibre/U$. Then, the fact that $Fix(\bhat)\not=\emptyset$ is a consequence of the real convexity theorem:
\begin{prop}[Lagrangian submanifolds of quasi-hamiltonian quotients]\label{applic1}
Let $(U,\tau)$ be a compact connected semi-simple Lie group endowed with an involutive automorphism $\tau$ of maximal rank. Let $\qhamsp$ be a quasi-hamiltonian $U$-space such that $\mu^{-1}(\{1\})\not=\emptyset$ and assume that $U$ acts freely on $\mu^{-1}(\{1\})$. Let $\beta:M\to M$ be an involution on $M$ satisfying:
\begin{enumerate}
\item[(i)] $\beta^*\w=-\w$.
\item[(ii)] $\beta(u.x)=\tau(u).\beta(x)$ for all $u\in U$ and all $x\in M$.
\item[(iii)] $\mu\circ \beta=\taum\circ\mu$.
\item[(iv)] $M^{\beta}\not=\emptyset$.
\item[(v)] $\mu(M^{\beta})$ intersects the connected component $Q_0$ of $1$ in $Fix(\taum)$.
\end{enumerate}
Then the involution $\bhat:[x]\in\fibre/U \mapsto [\beta(x)]$ is an anti-symplectic involution on the symplectic manifold $M//U=\qhamquot$, whose fixed-point set $Fix(\bhat)$ is non-empty. Consequently, $Fix(\bhat)$ is a lagrangian submanifold of $M//U$.
\end{prop}
\begin{proof}
Assume first that $U$ is simply connected. The involution $\beta$ satisfies the assumptions of theorem \ref{real_convexity_thm}. Consequently, by corollary \ref{fibres_and_fix_beta}, one has $\fibre\cap Fix(\beta)\not=\emptyset$, which implies $Fix(\bhat)\not=\emptyset$. When $U$ is semi-simple, we refer to \cite{Sch_AIF} to see how to reduce this situation to the case of the universal cover $\tU$ of $U$ (which is still compact as $U$ is assumed to be semi-simple).
\end{proof}
This general result becomes especially interesting when applied to the quasi-hamiltonian space $M=\Mtot$, where $U$ is a compact connected Lie group and $\calC_1,\, ...\, , \calC_j$ are $l$ conjugacy classes in $U$. Alekseev, Malkin and Meinrenken have shown in \cite{AMM} that in this case the quasi-hamiltonian quotient $M//U$ is the space of representations $\HomC(\pi_{g,l},U)/U$ of the fundamental group $\pi_{g,l}:=\pigl$ into $U$ (here, $\Sigma_g$ denotes a compact Riemann surface of genus $g\geq 0$ and $s_1,\, ...\, , s_l$ are $l$ pairwise distinct points of $\Sigma_g$). This fundamental group has the following finite presentation: $$\pigl=\prespiintro$$ and the momentum map of the diagonal action of $U$ on $M$ is:
\begin{eqnarray*}
\mu: \Mtot & \longto & U\\
\abc & \longmapsto & \relabc
\end{eqnarray*}
so that $\qhamquot$ is the set of equivalence classes of group morphisms $\rho:\pi_{g,l}\to U$ satisfying $\rho(\gamma_j)\in\calC_j$ for all $j\in\{1,\, ...\, , l\}$: $$\qhamquot=\HomC(\pi_{g,l},U)/U.$$ 
Then, by proposition \ref{applic1}, in order to find a lagrangian submanifold of the representation space $\HomC(\pi_{g,l},U)/U$, it is sufficient to give an example of an involution $\beta$ on the quasi-hamiltonian space $M=\Mtot$ satisfying the assumptions of the real convexity theorem \ref{real_convexity_thm}. There is nonetheless somehing to be careful about in this approach: usually the representation space $\HomC(\pi_{g,l},U)/U$ is not a manifold. But it is a stratified space in the sense of Lerman and Sjamaar in \cite{LS}. In particular, it is a disjoint union of symplectic manifolds called \emph{strata}. Then, as a generalization of proposition \ref{applic1}, a compatible form-reversing involution $\beta$ on $M=\Mtot$ induces an involution $Fix(\bhat)$ whose fixed-point set is non-empty (since proposition \ref{applic1} showed that we had in fact the stronger statement $\fibre\cap Fix(\beta)\not=\emptyset$) and is a union of lagrangian submanifolds of some of the strata. We refer to \cite{Sch_AIF} for a proof of this result and we now focus on the construction of such an involution $\beta$ on $\Mtot$. This construction relies on the notion of \emph{decomposable representation} of $\pi_{g,l}$ into $U$. A representation of $\pi_{g,l}$ into $U$ can be thought of as a $(2g+l)$-tuple $\abc\in \Mtot$ satisfying $\relabcshort=1$. To keep the exposition simple, we restrict ourselves here to the case where $g=0$, but the notion of decomposable representation, and therefore theorem \ref{applic2}, are valid for arbitrary $g\geq 0$ (see \cite{Sch_AIF}).  When $U=U(n)$, the notion of decomposable representation is due to Falbel and Wentworth (see \cite{FW}, and see \cite{Sch_CJM} and \cite{Sch_AIF} for arbitrary $U$ and arbitrary $g\geq 0$). In the $g=0$ case and for arbitrary $U$, the representation $(c_1,\, ...\, , c_l)\in\pconj$ is said to be decomposable if there exist elements $w_1, ..., w_l\in Fix(\taum)\subset U$ satisfying $c_1=w_1w_2^{-1},\, ...\, , c_l=w_lw_1^{-1}$. To characterize such decomposable representations, the idea is to introduce an involution $\beta$ on $M=\pconj$ (see \cite{Sch_CJM} for an explanation on the origin of this idea). This involution is:
\begin{defi}\label{def_beta}
\begin{eqnarray*}
\beta: \pconj & \longto & \pconj \\
(c_1,\, ...\, , c_l) & \longmapsto & (\taum(c_2...c_l)\taum(c_1)\tau(c_2...c_l),\, ...\, , \taum(c_l)\taum(c_{l-1})\tau(c_l), \taum(c_l))
\end{eqnarray*}
\end{defi}
This $\beta$ satisfies the assumptions of theorem \ref{real_convexity_thm} (as usual, the symmetric pair $(U,\tau)$ is assumed to be of maximal rank) and also characterizes decomposable representations in the following sense: a representation $(c_1,\, ...\, , c_l)$ is decomposable if and only if there exists an element $\pphi\in Fix(\taum)$ such that $\beta(c_1,\, ...\, , c_l)=\pphi.(c_1,\, ...\, , c_l)$ (see \cite{Sch_CJM} for the case $g=0$ and \cite{Sch_AIF} for the case of decomposable representations of $\pi_{g,l}$ with $g\geq 1$). Since $\beta$ satisfies the assumptions of the real convexity theorem \ref{real_convexity_thm}, corollary \ref{fibres_and_fix_beta} shows that $\fibre\cap Fix(\beta)\not=\emptyset$, which shows in particular that there always exist decomposable representations (since if $\beta(c_1,\, ...\, , c_l)=(c_1,\, ...\, , c_l)$, the representation $(c_1,\, ...\, , c_l)$ is decomposable). Consequently:
\begin{thm}[A lagrangian subspace of the representation space of a surface group]\cite{Sch_AIF}\label{applic2}
Let $\beta$ be the involution of $M=\Mtot$ characterizing the decomposable representations of $\pi_{g,l}$ into $U$. Then $\fibre\cap Fix(\beta)\not=\emptyset$ and therefore $Fix(\bhat)\not=\emptyset$ on $M//U=\HomC(\pi_{g,l},U)/U$. Consequently, $Fix(\bhat)$ is a lagrangian subspace of the representation space $\HomC(\pi_{g,l},U)/U$, containing the set of equivalence classes of decomposable representations.
\end{thm}
When $U$ is simply connected, the proof of theorem \ref{applic2} is an immediate application of the real convexity theorem. We refer to \cite{Sch_AIF} for the case where the compact connected Lie group $U$ is not simply connected. We now end this paper with an application to linear algebra of the real convexity theorem \ref{real_convexity_thm}:
\begin{thm}\cite{Sch_AIF}\label{applic3}
Consider $\lambda_1, ...\, ,$ $\lambda_l\in\R^n$. Then the following
statements are equivalent:
\begin{enumerate}
\item[(i)] There exist $l$ special unitary
matrices $u_,\, ...\, ,u_l\in SU(n)$ such that:
$$\mathrm{Spec}\,u_j=\exp(i\lambda_j)\quad and \quad u_ l...u_l=1.$$
\item[(ii)] There exist $l$ special unitary matrices $A_1,\, ...\, ,A_l \in SU(n)$ such that: $$\mathrm{Spec}\, (A_j^tA_j) =\exp(i\lambda_j)\quad
and\quad A_1...A_l=1.$$
\end{enumerate}
\end{thm}
\begin{proof}
Let $\calC_j$ be the conjugacy class of $\exp(i\lambda_j)\in SU(n)$. Assume first that $A_1,\, ...\, ,A_l\in SU(n)$ satisfy $A_j^tA_j\in \calC_j$ and $A_1...A_l=1$. Set $u_l=A_l^tA_l$, $u_{l-1}=A_l^t(A_{l-1}^tA_{l-1})(A_l^t)^{-1}, ...\, ,$ and $u_1=(A_2...A_l)^t(A_1^tA_1)((A_2...A_l)^t)^{-1}$. Then $u_j\in\calC_j$ for all $j$ and
\begin{eqnarray*}
u_1...u_l & = & (A_2...A_l)^t(A_1)^tA_1(A_2...A_l)\\
& = & (A_1...A_l)^t(A_1...A_l)\\
& = & 1
\end{eqnarray*}
which proves that (ii) implies (i).\\
Conversely, consider $u_1,\, ...\, ,u_l\in SU(n)$ satisfying $u_j\in\calC_j$ and $u_1...u_l=1$. This means that the momentum map
\begin{eqnarray*}
\mu:\pconj & \longto & SU(n)\\
(u_1,\, ...\, ,u_l) & \longmapsto & u_1...u_l
\end{eqnarray*}
has a non-empty fibre above $1\in SU(n)$. Let $\beta$ be the involution defined in \ref{def_beta} and let 
$\tau$ be the involutive automorphism $\tau:u\mapsto\overline{u}$ of $SU(n)$. This involution is of maximal rank since the maximal torus of $SU(n)$ consisting of diagonal matrices is fixed pointwise by $\taum:u\mapsto u^t$. By corollary \ref{fibres_and_fix_beta}, $\mu^{-1}(\{1\})\cap Fix(\beta)\not=\emptyset$.
Consider then $(w_1,\, ...\, ,w_l)\in Fix(\beta)\cap\mu^{-1}(\{1\})$. Then $\beta(w_1,\, ...\, ,w_l)=(w_1,\, ...\, ,w_l)$, that is:
\begin{eqnarray*}
w_l^t & = & w_l\\
w_l^t w_{l-1}^t (w_l^t)^{-1} & = & w_{l-1}\\
& \vdots &\\
\big(w_2...w_l\big)^tw_1^t\big((w_2...w_l)^t\big)^{-1} & = & w_1
\end{eqnarray*}
Since $Fix(\taum)\subset SU(n)$ is connected (every symmetric unitary matrix is of the form $w=\exp(iS)$ where $S$ is a real symmetric matrix), we can write $w_l=A_l^tA_l$ for some $A\in SU(n)$ (take for instance $A=\exp(i\frac{S}{2}))$. Using the above equations, we can then write $A_l^tA_lw_{l-1}^tA_l^{-1}(A_l^t)^{-1}=w_{l-1}$, that is: $$\big((A_l^t)^{-1}w_{l-1}A_l^t\big)^t=(A_l^t)^{-1}w_{l-1}A_l^t$$ and we can therefore write $$(A_l^t)^{-1}w_{l-1}A_l^t=A_{l-1}^tA_{l-1}$$ for some $A_{l-1}\in SU(n)$. Continuing like this, we obtain, for all $j\in\{2,\, ...\, ,n\}$:
\begin{eqnarray}
\big((A_{j+1}...A_l)^t\big)^{-1}w_j\big(A_{j+1}...A_l\big)^t & =  & A_j^tA_j.\label{symrel}
\end{eqnarray}
In particular, $A_j^tA_j\in\calC_j$ for all $j\geq 2$. We then set $A_1:=(A_2...A_l)^{-1}$. Then:
\begin{eqnarray*}
& & \big(A_2...A_l\big)^t\big(A_1^tA_1\big)\big((A_2...A_l)^t\big)^{-1}\\
& = & \big(A_2...A_l\big)^t\big((A_2...A_l)^{-1}\big)^t \big(A_2...A_l)^{-1} \big((A_2...A_l)^t\big)^{-1}\\
& = & \big((A_2...A_l)^t(A_2...A_l)\big)^{-1}\\
& = & \big(A_l^t...A_3^t(A_2^tA_2)A_3...A_l\big)^{-1}\\
\mathrm{\big(using\ (}\ref{symrel}\mathrm{)\big)} & = & \big(\underbrace{A_l^t...A_3^t\big((A_3...A_l)^t\big)^{-1}}_{=1}\, w_2 \, (A_3...A_l)^tA_3...A_l\big)^{-1}\\
\mathrm{\big(by\ induction\big)} & = & \big(w_2w_3...w_l)^{-1}\\
& = & w_1
\end{eqnarray*}
since $w_1...w_l=1$. In particular, $A_1^tA_1$ is conjugate to $w_1$ and therefore $A_1^tA_1\in\calC_1$. Since $A_1...A_l=1$ by definition of $A_1$, this shows that (i) implies (ii). 
\end{proof}
As a matter of fact, the above theorem is also true for $U=U(n)$, as shown in \cite{Sch_AIF}. In that case too, the proof is ultimately a corollary of the real convexity theorem \ref{real_convexity_thm}.


\end{document}